\documentclass[12 pt]{article}%
\usepackage{amsmath, amsfonts, amsthm, color,latexsym}
\usepackage{amsmath}
\usepackage{amsfonts}
\usepackage{amssymb}
\usepackage{color, soul}
\usepackage[all]{xy}
\usepackage{graphicx}%
\providecommand{\U}[1]{\protect\rule{.1in}{.1in}}
\allowdisplaybreaks[4]
\newtheorem{theorem}{Theorem}[section]
\newtheorem{proposition}[theorem]{Proposition}
\newtheorem{corollary}[theorem]{Corollary}
\newtheorem{example}[theorem]{Example}

\newtheorem{remark}[theorem]{Remark}

\newtheorem{lemma}[theorem]{Lemma}
\newtheorem{final remark}[theorem]{Final Remark}

\allowdisplaybreaks[4]

\begin{document}

\title{Bidual extensions of Riesz multimorphisms}
\author{Geraldo Botelho\thanks{Supported by CNPq Grant
304262/2018-8 and Fapemig Grant PPM-00450-17.}\,\, and  Luis Alberto Garcia\thanks{Supported by a CAPES scholarship.\newline 2020 Mathematics Subject Classification: 46A40, 46B40, 46B42, 46G25.\newline Keywords: Riesz spaces, Riesz multimorphism, Arens extension, Aron-Berner extension, Banach lattices.
}}
\date{}
\maketitle

\begin{abstract} We prove that all Arens extensions of finite rank Riesz multimorphisms taking values in Archimedean Riesz spaces coincide and are Riesz multimorphisms. Partial results for arbitrary Riesz multimorphisms are obtained. We also prove that, for a class of Banach lattices $F$, which includes $F = c_0, \ell_p, c_0(\ell_p), \ell_p(c_0), \ell_p(\ell_s), 1 < p,s < \infty$, among many others, all Aron-Berner extensions of $F$-valued Riesz multimorphisms between Banach lattices are Riesz multimorphisms.
\end{abstract}

\section{Introduction}

It is well known that the second adjoint of a Riesz homomorphism $u$ between Riesz spaces or Banach lattices is a bidual extension of $u$ and a Riesz homomorphism as well. For multilinear operators, bidual extensions of bilinear operators between normed spaces were first considered by Arens \cite{arens}. In this paper we investigate if Arens extensions of Riesz multimorphisms between Riesz spaces or Banach lattices are Riesz multimorphisms.

By $E^\sim$ we denote the order dual of a Riesz space $E$, hence $E^{\sim\sim} = (E^\sim)^\sim$ denotes its second order dual. If $E$ is a Banach lattice, $E^*$ denotes is topological dual, hence $E^{**}$ stands for its bidual. The symbols $(E^\sim)_n^{\sim}$ and $(E^*)_n^{*}$ stand for the corresponding subspaces formed by the order continuous functionals.

Given Banach lattices $E_1, E_2,F$, and a Riesz bimorphism $A \colon E_1 \times E_2 \longrightarrow F$, Scheffold \cite{scheffold} proved that the restriction of the Arens extension $A^{***} \colon E_1^{**} \times  E_2^{**} \longrightarrow F^{**}$ to $(E_1^*)_n^{*} \times (E_2^*)_n^{*}$ is a Riesz bimorphism. More recently, Boulabiar, Buskes and Page \cite{boulabiar1} proved that, given Archimedean Riesz spaces $E_1, E_2, F$ and a Riesz bimorphism $A \colon E_1 \times E_2 \longrightarrow F$, the restriction of the Arens extension $A^{***} \colon E_1^{\sim\sim} \times  E_2^{\sim\sim} \longrightarrow F^{\sim\sim}$ to $(E_1^\sim)_n^{\sim} \times (E_2^\sim)_n^{\sim }$ is a Riesz bimorphism. Multilinear versions of such extensions in the Banach space setting have been studied for a long time (see, e.g. \cite{livrosean}), usually under the name of Aron-Berner extensions \cite{aronberner}. In the lattice environment, an Arens extension $A^{*(m+1)}\colon E_1^{\sim\sim} \times \cdots \times E_m^{\sim\sim} \longrightarrow F^{\sim\sim}$ of a regular $m$-linear operator $A \colon E_1 \times \cdots \times E_m \longrightarrow F$ between Riesz spaces was studied by Buskes and Roberts \cite{Buskes}, and the corresponding Aron-Berner extension $A^{*(m+1)}\colon E_1^{**} \times \cdots \times E_m^{**} \longrightarrow F^{**}$ of a regular $m$-linear operator $A \colon E_1 \times \cdots \times E_m \longrightarrow F$ between Banach lattices was studied by Boyd, Ryan and Snigireva \cite{ryan1}. In the perspective of these very recent developments, it is natural to address the generalization of the results of Scheffold and Boulabiar et al. in three directions: considering $m$-linear operators, $m \in \mathbb{N}$; taking into account all $m!$ Arens extensions of an $m$-linear operator; investigating if the extensions of Riesz multimorphisms are Riesz multimorphisms on the product of the whole biduals, and not only on the product of the corresponding restrictions. More precisely, the general problems are the following: 

\noindent  (i) Given an $m$-linear Riesz multimorphism $A \colon E_1 \times \cdots \times E_m \longrightarrow F$ between Riesz spaces, are all $m!$ Arens extensions of $A$  Riesz multimorphisms on $E_1^{\sim\sim} \times \cdots \times E_m^{\sim\sim}$?\\
  (ii) Given an $m$-linear Riesz multimorphism $A \colon E_1 \times \cdots \times E_m \longrightarrow F$ between Banach lattices, are all $m!$ Aron-Berner extensions of $A$ Riesz multimorphisms on $E_1^{**} \times \cdots \times E_m^{**}$?

In order to tackle these problems, we develop in Section \ref{section2} the theory of Arens extensions of regular multilinear operators on Riesz spaces following the approach Cabello S\'anchez, Garc\'ia and Villanueva \cite{extracta} applied for Banach spaces. We obtain, in particular, an analogue of the Davie and Gamelin \cite{gamelin} description of the Aron-Berner extensions on Banach spaces for the Arens extensions on Riesz spaces, where the weak-star topology is replaced with the weak absolute topology. Using the approach and the results of Section \ref{section2}, we prove the results on extensions of Riesz multimorphisms in Section \ref{section 3}. First we extend the aforementioned result due Boulabiar et al., with a different technique, to the multilinear case and to all different Arens extensions of a Riesz multimorphism. Then we show that question (i) holds true for finite rank Riesz multimorphisms taking values in Archimedean Riesz spaces. Moreover, we prove that all Arens extensions of such a Riesz multimorphism coincide. Some partial results for the general case, which will be helpful later, are also proved. Finally we prove that  question (ii) holds true for Riesz multimorphisms taking values in a class of Banach lattices $F$ which includes, for instance, $F= c_0, \ell_p,$ Tsirelson's original space $ T^*$, its dual $T$, Schreier's space $ S$, the predual $d_*(w,1)$ of the Lorenz sequence space $d(w,1)$, $c_0(E)$ and $\ell_p(E)$ where $1 < p < \infty$ and $E = c_0, \ell_s, 1< s < \infty, T^*, T, S, d_*(w,1)$.

As usual, by $J_E \colon E \longrightarrow E^{\sim\sim}$ we denote the canonical operator ($J_E(x)(x'') = x''(x)$), which happens to be a Riesz homomorphism. If $E$ is a Banach space, $J_E \colon E \longrightarrow E^{**}$ is the canonical embedding. Given Riesz spaces $E_1, \ldots, E_m,F$, the space of regular $m$-linear operators from $E_1 \times \cdots \times E_m$ to $F$ is denoted by ${\cal L}_r(E_1, \ldots, E_m;F)$. When $F$ is the scalar field we write ${\cal L}_r(E_1, \ldots, E_m)$.
Recall that a Riesz multimorphism is an operator $A \in {\cal L}_r(E_1, \ldots, E_m;F)$ such that
$$|A(x_1, \ldots, x_m)| = A(|x_1|, \ldots, |x_m|) $$
for all $x_1 \in E_1, \ldots, x_m \in E_n$ (details can be found in \cite{boulabiar}). For the theory of regular multilinear operators we refer to \cite{bu, loane}.

\section{Arens extensions of regular multilinear operators}\label{section2}

In this section we apply the method of Cabello S\'anchez, Garc\'ia and Villanueva \cite{extracta} to construct the $m!$ Arens extensions of a regular $m$-linear operator. Given $m \in \mathbb{N}$, by $S_m$ we denote the set of permutations of $\{1, \ldots, m\}$. Given Riesz spaces $E_1,\ldots, E_m$, a permutation $\rho\in S_{m}$ and $k\in \{1,\ldots,m\}$, we fix the following notation:
$$E_{1},\ldots,\,_{\rho(1)}E,\ldots,\,_{\rho(k-1)}E,\ldots,E_{m}=\left\{ \begin{array}{cl}
E_{1},\ldots, E_{m} \mbox{~in this order} & \mbox{if}\,\ k=1, \\
E_{1},\ldots, E_{m} \mbox{~in this order, where}\\ E_{\rho(1)},\ldots,
   E_{\rho(k-1)} \mbox{~are removed} & \mbox{if}~ k=2,\ldots,m.
\end{array}\right.$$
For instance, $(E_1,\,_2E, E_3) = (E_1, E_3)$.
The same procedure defines the $(m-k+1)$-tuple $(x_{1},\ldots,\,_{\rho(1)}x,\ldots,\,_{\rho(k-1)}x,\ldots,x_{m})$ and the cartesian product $E_{1}\times\cdots \times\,_{\rho(1)}E \times\cdots \times\,_{\rho(k-1)}E\times\cdots \times E_{m}$.   Moreover, for $k=1,\ldots,m-1$, we write $$E_{1},\ldots,\,_{\rho(1)}E,\ldots,\,_{\rho(k)}E,\ldots,E_{m}=E_{1},\ldots, E_{m}$$ in this order, where $E_{\rho(1)},\ldots,E_{\rho(k)}$ are removed. In the same fashion we define the $(m-k)$-tuple $(x_{1},\ldots,_{\rho(1)}x,\ldots,_{\rho(k)}x,\ldots,x_{m})$ and the corresponding cartesian product.

  Finally, for $k=m$ we write $\mathcal{L}(E_{1},\ldots,\,_{\rho(1)}E,\ldots,\,_{\rho(k)}E,\ldots,E_{m};\mathbb{R})=\mathbb{R}.$

 Let $k\in\{1,\ldots,m\}$, a permutation $\rho\in S_{m}$, Riesz spaces $E_{1},\ldots,E_{m}$ and an operator $A\in\mathcal{L}_{r}(E_{1},\ldots,\,_{\rho(1)}E,\ldots,\,_{\rho(k-1)}E,\ldots,E_{m})$ be given. For $x_{r}\in E_{r}, r\in \{1,\ldots,m\}\setminus \{\rho(1),\ldots,\rho(k)\}$, consider the linear functionals
$$A(x_{1},\ldots,\,_{\rho(1)}x,\ldots,\,_{\rho(k)}x;\bullet\,;\ldots, x_{m})\colon E_{\rho(k)}\longrightarrow \mathbb{R},$$
\begin{equation}\label{defpunto}
A(x_{1},\ldots,\,_{\rho(1)}x,\ldots,\,_{\rho(k)}x;\bullet\,;\ldots,x_{m})(x_{\rho(k)})
=A(x_{1},\ldots,\,_{\rho(1)}x,\ldots,\,_{\rho(k-1)}x,\ldots,x_{m}),
\end{equation}
where the dot $\bullet$ is placed at the $\rho(k)$-th coordinate.  Observe that for $k=m$ we have $A(x_{1},\ldots,\,_{\rho(1)}x,\ldots,\,_{\rho(m)}x;\bullet\,;\ldots,x_{m})=A\in E_{\rho(m)}^{\sim\sim}$. These are called functionals associated to the operator $A$. 

\begin{proposition}\label{operadores positivos}
 Let $E_{1},\ldots,E_{m}$ be Riesz spaces, $\rho\in S_{m}$, $k \in \{1,\ldots,m\}$ and $x_{\rho(k)}^{\prime\prime} \in E_{\rho(k)}^{\sim\sim}$. The operator
$$\overline{x_{\rho(k)}^{\prime\prime}}^{\rho}\colon \mathcal{L}_{r}(E_{1},\ldots,_{\rho(1)}E,\ldots,_{\rho(k-1)}E,\ldots,E_{m})\longrightarrow \mathcal{L}_{r}(E_{1},\ldots,_{\rho(1)}E,\ldots,_{\rho(k)}E,\ldots,E_{m}),$$
$$\overline{x_{\rho(k)}^{\prime\prime}}^{\rho}(A)(x_{1},\ldots,_{\rho(1)}x,
\ldots,_{\rho(k)}x,\ldots,x_{m})=x_{\rho(k)}^{\prime\prime}(A(x_{1},\ldots,_{\rho(1)}x,\ldots,_{\rho(k)}x;\bullet\,;\ldots,x_{m})),$$ 
is linear, regular and $\Big|\overline{x_{\rho(k)}^{\prime\prime}}^{\rho} \Big|\leq \overline{|x_{\rho(k)}^{\prime\prime}|}^{\rho}$. Furthermore, if $0\leq x_{\rho(k)}^{\prime\prime} \in E_{\rho(k)}^{\sim\sim}$ then the operator $\overline{x_{\sigma(k)}^{\prime\prime}}^{\rho}$ is positive.
\end{proposition}
\begin{proof}
It is easy to see that, for each $A\in \mathcal{L}_{r}(E_{1},\ldots,\,_{\rho(1)}E,\ldots,\,_{\rho(k-1)}E,\ldots,E_{m})$, 
$\overline{x_{\rho(k)}^{\prime\prime}}^{\rho}(A)$ is $(m-k)$-linear. 
Let $A_{1},A_{2}\in  \mathcal{L}(E_{1},\ldots,\,_{\rho(1)}E,\ldots,\,_{\rho(k-1)}E,\ldots,E_{m})$ be positive operators such that $A=A_{1}-A_{2}.$  For every $x_{j}\in E_{j}, j\in \{1,\ldots,m\}\setminus \{\rho(1),\ldots,\rho(k)\}$, call $$y=(x_{1},\ldots,\,_{\rho(1)}x,\ldots,\,_{\rho(k)}x,\ldots,x_{m})\in E_{1}\times\cdots \times_{\rho(1)}E \times \cdots\times _{\rho(k)}E\times\cdots \times E_{m},$$ $\varphi=A(x_{1},\ldots,\,_{\rho(1)}x,\ldots,\,_{\rho(k)}x;\bullet\,;\ldots,x_{m})$ and $\varphi_{i}=A_{i}(x_{1},\ldots,\,_{\rho(1)}x,\ldots,\,_{\rho(k)}x;\bullet\,;\ldots,x_{m})$ for $i=1,2$. Then $\varphi=\varphi_{1}-\varphi_{2}$ because $A=A_{1}-A_{2},$ therefore
\begin{align*}
\overline{x_{\rho(k)}^{\prime\prime}}^{\rho}&(A)(y) = x_{\rho(k)}^{\prime\prime}(\varphi)=\big((x_{\rho(k)}^{\prime\prime})^{+}-(x_{\rho(k)}^{\prime\prime})^{-}\big)(\varphi)=(x_{\rho(k)}^{\prime\prime})^{+}(\varphi)-(x_{\rho(k)}^{\prime\prime})^{-}(\varphi)\\
&=\big((x_{\rho(k)}^{\prime\prime})^{+}(\varphi_{1})+(x_{\rho(k)}^{\prime\prime})^{-}(\varphi_{2})\big)-\big((x_{\rho(k)}^{\prime\prime})^{+}(\varphi_{2})+(x_{\rho(k)}^{\prime\prime})^{-}(\varphi_{1})\big)\\
&=\Big(\overline{(x_{\rho(k)}^{\prime\prime})^{+}}^{\rho}(A_{1})(y)+\overline{(x_{\rho(k)}^{\prime\prime})^{-}}^{\rho}(A_{2})(y)\Big)-\Big(\overline{(x_{\rho(k)}^{\prime\prime})^{+}}^{\rho}(A_{2})(y)+\overline{(x_{\rho(k)}^{\prime\prime})^{-}}^{\rho}(A_{1})(y)\Big)\\
&=\Big(\overline{(x_{\rho(k)}^{\prime\prime})^{+}}^{\rho}(A_{1})+\overline{(x_{\rho(k)}^{\prime\prime})^{-}}^{\rho}(A_{2})\Big)(y)-\Big(\overline{(x_{\rho(k)}^{\prime\prime})^{+}}^{\rho}(A_{2})+\overline{(x_{\sigma(k)}^{\prime\prime})^{-}}^{\rho}(A_{1})\Big)(y),
\end{align*}
proving that
 $$\overline{x_{\rho(k)}^{\prime\prime}}^{\rho}(A)=\underbrace{\Big(\overline{(x_{\rho(k)}^{\prime\prime})^{+}}^{\rho}(A_{1})+\overline{(x_{\rho(k)}^{\prime\prime})^{-}}^{\rho}(A_{2})\Big)}_{:=T_1}-\underbrace{\Big(\overline{(x_{\rho(k)}^{\prime\prime})^{+}}^{\rho}(A_{2})
 +\overline{(x_{\rho(k)}^{\prime\prime})^{-}}^{\rho}(A_{1})\Big)}_{:=T_2}.$$
It is plain that $T_1$ and $T_2$ are $(m-k)$-linear operators. 
 For $x_{j}\in E_{j}^{+}$, $j\in \{1,\ldots,m\}\setminus \{\rho(1),\ldots,\rho(k)\},$  the functionals associated to $A_{1}$ and $A_{2}$,
$$A_{i}(x_{1},\ldots,\,_{\rho(1)}x,\ldots,\,_{\rho(k)}x;\bullet\,;\ldots,x_{m})\colon E_{\rho(k)}\longrightarrow \mathbb{R},\, i=1,2,$$
are positive. Since $(x_{\rho(k)}^{\prime\prime})^{+}$ and  $(x_{\rho(k)}^{\prime\prime})^{-}$ are positive functionals,  we have, for $i=1,2$, $\scriptsize{(x_{\rho(k)}^{\prime\prime})^{\pm}(A_{i}(x_{1},\ldots,\,_{\rho(1)}x,\ldots,\,_{\rho(k)}x;\bullet\,;\ldots,x_{m}))} $ are positive in $\mathbb{R}$
for $i=1,2$. It follows that
$\footnotesize{\overline{(x_{\rho(k)}^{\prime\prime})^{\pm}}^{\rho}(A_{i})(x_{1},\ldots,\,_{\rho(1)}x,\ldots,\,_{\rho(k)}x,\ldots,x_{m})}$
   are positive, so   $\footnotesize{\overline{(x_{\rho(k)}^{\prime\prime})^{+}}}^{\rho}(A_{i})\geq 0$ and $\footnotesize{\overline{(x_{\rho(k)}^{\prime\prime})^{-}}}^{\rho}(A_{i}) \geq 0$, $i=1,2$. We conclude that $T_{i}\geq 0, i=1,2$, which proves that  $\overline{x_{\rho(k)}^{\prime\prime}}^{\rho}(A)$ is regular, that is, $\overline{x_{\rho(k)}^{\prime\prime}}^{\rho}$ is well defined. The linearity is clear. Now let  $A\in \mathcal{L}_{r}(E_{1},\ldots,\,_{\rho(1)}E,\ldots,\,_{\rho(k-1)}E,\ldots,E_{m})$ be positive and $0\leq x_{\rho(k)}^{\prime\prime} \in E_{\rho(k)}^{\sim\sim}$ be given. For all $x_{j}\in E_{r}^{+}, j\in \{1,\ldots,m\}\setminus \{\rho(1),\ldots,\rho(k)\}$, the functionals associated to $A$ are positive, hence
\begin{align*}
0\leq x_{\rho(k)}^{\prime\prime}(A(x_{1},\ldots,\,_{\rho(1)}x,\ldots,\,_{\rho(k)}x;\bullet\,;\ldots,x_{m}))\hspace{-0.1cm}
=\overline{x_{\rho(k)}^{\prime\prime}}^{\rho}(A)(x_{1},\ldots,\,_{\rho(1)}x,\ldots,\,_{\rho(k)}x,\ldots,x_{m}),
\end{align*}
showing that $\overline{x_{\rho(k)}^{\prime\prime}}^{\rho}$ is positive. Note that $\overline{x^{\prime\prime}_{\rho(k)}}^{\rho}=\overline{(x^{\prime\prime}_{\rho(k)})^{+}}^{\rho} - \overline{(x^{\prime\prime}_{\rho(k)})^{-}}^{\rho}$, so, since $(x^{\prime\prime}_{\rho(k)})^{+}$, $(x^{\prime\prime}_{\rho(k)})^{-}\in E_{\rho(k)}^{\sim\sim}$ are positive, by what we did above we get that $\overline{(x^{\prime\prime}_{\rho(k)})^{+}}^{\rho}$ and $\overline{(x^{\prime\prime}_{\rho(k)})^{-}}^{\rho} $ are positive functionals, proving that $\overline{x^{\prime\prime}_{\rho(k)}}^{\rho}$ is regular.

To finish the proof, 
  let $x_{j}\in E_{j}^{+}$, $j\in \{1,\ldots,m\}\setminus \{\rho(1),\ldots,\rho(k)\}$ and let $A\in \mathcal{L}_{r}(E_{1},\ldots,\,_{\rho(1)}E,\ldots,\,_{\rho(k-1)}E,\ldots,E_{m})$ be a positive operator. Since the functionals associated to $A$ are positive,
\begin{align*}
\overline{|x_{\rho(k)}^{\prime\prime}|}^{\rho}(A)(x_{1},\ldots,\,_{\rho(1)}x,\ldots,\,_{\rho(k)}x,&\ldots,x_{m})=
|x_{\rho(k)}^{\prime\prime}|(A(x_{1},\ldots,\,_{\rho(1)}x,\ldots,\,_{\rho(k)}x;\bullet\,;\ldots,x_{m}))\\
&\geq \pm x_{\rho(k)}^{\prime\prime}(A(x_{1},\ldots,\,_{\rho(1)}x,\ldots,\,_{\rho(k)}x;\bullet\,;\ldots,x_{m}))\\
&=\pm \Big(\overline{x_{\rho(k)}^{\prime\prime}}^{\rho}\Big)(A)(x_{1},\ldots,\,_{\rho(1)}x,\ldots,\,_{\rho(k)}x,\ldots,x_{m}),
\end{align*}
from which we get $\overline{|x_{\rho(k)}^{\prime\prime}|}^{\rho}\geq \pm \Big(\overline{x_{\rho(k)}^{\prime\prime}}^{\rho}\Big)$. It follows that  $\overline{|x_{\rho(k)}^{\prime\prime}|}^{\rho}\geq \Big| \overline{x_{\rho(k)}^{\prime\prime}}^{\rho}\Big|$.
\end{proof}

In \cite{ryan1, Buskes} the authors applied the technique of Arens \cite{arens} to construct a bidual extension of a regular multilinear operators, which we describe now. 
Given an $m$-linear operator $A\colon E_{1}\times\cdots\times E_{m}\longrightarrow F$ between Riesz spaces, consider the following $m$-linear regular operators:\\
$\circ ~A^{\ast}\colon F^{\sim}\times E_{1}\times\cdots\times E_{m-1}\longrightarrow E_{m}^{\sim}~,~ A^{\ast}(y^{\prime},x_{1},\ldots,x_{m-1})(x_{m})=y^{\prime}(A(x_{1},\ldots,x_{m})),$

\noindent $\circ~ A^{\ast\ast}\colon E_{m}^{\sim\sim}\times F^{\sim}\times E_{1}\times\cdots\times E_{m-2}\longrightarrow E_{m-1}^{\sim},$
$$ A^{\ast\ast}(x_{m}^{\prime\prime},y^{\prime},x_{1},\ldots,x_{m-2})(x_{m-1})=x_{m}^{\prime\prime}(A^{\ast}(y^{\prime},x_{1},\ldots,x_{m-1})),$$
$\circ ~A^{\ast\ast\ast}\colon E_{m-1}^{\sim\sim}\times E_{m}^{\sim\sim}\times F^{\sim}\times E_{1}\times\cdots\times E_{m-3}\longrightarrow E_{m-2}^{\sim},$
$$ A^{\ast\ast\ast}(x_{m-1}^{\prime\prime},x_{m}^{\prime\prime},y^{\prime},x_{1},\ldots,x_{m-3})(x_{m-2})=x_{m-1}^{\prime\prime}(A^{\ast\ast}(x_{m}^{\prime\prime},y^{\prime},x_{1},\ldots,x_{m-2})),$$

\vspace{-1.0em}

$\,\,\vdots$

\medskip

\noindent $\circ ~A^{\ast(m+1)}\colon E_{1}^{\sim\sim}\times\cdots\times E_{m}^{\sim\sim}\longrightarrow F^{\sim\sim},\, A^{\ast(m+1)}(x_{1}^{\prime\prime},\ldots,x_{m}^{\prime\prime})(y^{\prime})=x_{1}^{\prime\prime}
(A^{\ast(m)}(x_{2}^{\prime\prime},\ldots,x_{m}^{\prime\prime},y^{\prime})).$

\medskip

By $\theta$ we denote the backward shift permutation of $\{1, \ldots, m\}$, that is, $\theta(m) = m-1,  \theta(m-1) = m-2, \ldots,\theta(2) = 1, \theta(1) = m$.

Bearing in mind the operators $\overline{x''_{\rho(k)}}^\rho$ of the previous proposition, we define the $m!$ Arens extensions of a regular $m$-linear operator and prove their basic properties.

\begin{theorem}\label{propriedadesherdadas}
Let $E_{1},\ldots,E_{m},F$ be Riesz spaces, $\rho\in S_{m}$ and $A \colon E_{1}\times\cdots\times E_{m} \longrightarrow F$ be an $m$-linear regular operator. Define $AR_{m}^{\rho}(A)\colon E_{1}^{\sim\sim}\times
\cdots\times E_{m}^{\sim\sim}\longrightarrow F^{\sim\sim}$ by $$ AR_{m}^{\rho}(A)(x_{1}^{\prime\prime},\ldots,x_{m}^{\prime\prime})(y^{\prime})=\big(\overline{x_{
\rho(m)}^{\prime\prime}}^{\rho}\circ\cdots\circ\overline
{x_{\rho(1)}^{\prime\prime}}^{\rho}\big)(y^{\prime}\circ A)$$
for every $y^{\prime}\in F^{\sim}$. Then:\\
{\rm (a)} $AR_{m}^{\rho}(A)$ is a regular $m$-linear operator.\\
{\rm (b)} $AR_{m}^{\rho}(A)$ extends $A$ in the sense that $AR_{m}^{\rho}(A)\circ (J_{E_{1}},\ldots,J_{E_{m}})=J_{F}\circ A.$\\
{\rm (c)} If  $A$ is positive, then $AR_{m}^{\rho}(A)$ is positive.\\
\noindent{\rm (d)} $AR_m^\theta(A) = A^{*(m+1)}$. In particular, $AR_2^\theta(A) = A^{***}$ in the bilinear case $m = 2$.
\end{theorem}

\begin{proof}
  Let us see that $AR_{m}^{\rho}(A)$ is well defined. Given $x_{1}^{\prime\prime}\in E_{1}^{\sim\sim},\ldots,x_{m}^{\prime\prime}\in E_{m}^{\sim\sim}$, it is clear that $AR_{m}^{\rho}(A)(x_{1}^{\prime\prime},\ldots,x_{m}^{\prime\prime})$ belongs to the algebraic bidual of $F$, so it is linear. We have to shows that it is regular. Since each $E_{j}^{\sim\sim}$ is a Riesz space, we have  $x_{j}^{\prime\prime}=(x_{j}^{\prime\prime})^{+}-(x_{j}^{\prime\prime})^{+}$, $j=1,\ldots,m$. For each $\xi^{m}=(\xi_{1},\ldots,\xi_{m}) \in \{+1,-1\}^m$,  write $P(\xi^{m})=\xi_{1}\cdots \xi_{m}\in \{+1,-1\}$. 
 For every $y^{\prime}\in F^{\sim}$, putting $U_{k}=\overline{x_{
\rho(m)}^{\prime\prime}}^{\rho}\circ\cdots\circ \overline{x_{
\rho(k)}^{\prime\prime}}^{\rho}$, $k=2,\ldots,m$, and using that $\overline{x_{
\rho(j)}^{\prime\prime}}^{\rho}=\overline{(x_{
\rho(j)}^{\prime\prime})^{+}-(x_{
\rho(j)}^{\prime\prime})^{-}}^{\rho}=\overline{(x_{
\rho(j)}^{\prime\prime})^{+}}^{\rho}-\overline{(x_{
\rho(j)}^{\prime\prime})^{-}}^{\rho}$, $j=1,\ldots,m$, we get
\begin{align*}
A&R_{m}^{\rho}(A)(x_{1}^{\prime\prime},\ldots,x_{m}^{\prime\prime})(y^{\prime})=\Big(\overline{x_{
\rho(m)}^{\prime\prime}}^{\rho}\circ\cdots\circ\overline
{x_{\rho(1)}^{\prime\prime}}^{\rho}\Big)(y^{\prime}\circ A)=U_{2}\Big(\overline
{x_{\rho(1)}^{\prime\prime}}^{\rho}(y^{\prime}\circ A)\Big)\\
&=U_{2}\Big(\Big(\overline{(x_{
\rho(1)}^{\prime\prime})^{+}}^{\rho}-\overline{(x_{
\rho(1)}^{\prime\prime})^{-}}^{\rho}\Big)(y^{\prime}\circ A)\Big)\\
&=U_{2}\Big(\overline{(x_{
\rho(1)}^{\prime\prime})^{+}}^{\rho}(y^{\prime}\circ A)\Big)-U_{2}\Big(\overline{(x_{
\rho(1)}^{\prime\prime})^{-}}^{\rho}(y^{\prime}\circ A)\Big)\\
&=U_{3}\Big(\overline{x_{
\rho(2)}^{\prime\prime}}^{\rho}\Big(\overline{(x_{
\rho(1)}^{\prime\prime})^{+}}^{\rho}(y^{\prime}\circ A)\Big)\Big)-U_{3}\Big(\overline{x_{
\rho(2)}^{\prime\prime}}^{\rho}\Big(\overline{(x_{
\rho(1)}^{\prime\prime})^{-}}^{\rho}(y^{\prime}\circ A)\Big)\Big)\\
&=U_{3}\Big(\Big(\overline{(x_{
\rho(2)}^{\prime\prime})^{+}}^{\rho}-\overline{(x_{
\rho(2)}^{\prime\prime})^{-}}^{\rho}\Big)\Big(\overline{(x_{
\rho(1)}^{\prime\prime})^{+}}^{\rho}(y^{\prime}\circ A)\Big)\Big)\\
&\quad \,\, -U_{3}\Big(\Big(\overline{(x_{
\rho(2)}^{\prime\prime})^{+}}^{\rho}-\overline{(x_{
\rho(2)}^{\prime\prime})^{-}}^{\rho}\Big)\Big(\overline{(x_{
\rho(1)}^{\prime\prime})^{-}}^{\rho}(y^{\prime}\circ A)\Big)\Big)\\
&=U_{3}\Big(\overline{(x_{
\rho(2)}^{\prime\prime})^{+}}^{\rho}\Big(\overline{(x_{
\rho(1)}^{\prime\prime})^{+}}^{\rho}(y^{\prime}\circ A)\Big)\Big)-U_{3}\Big(\overline{(x_{
\rho(2)}^{\prime\prime})^{-}}^{\rho}\Big(\overline{(x_{
\rho(1)}^{\prime\prime})^{+}}^{\rho}(y^{\prime}\circ A)\Big)\Big)\\
&\quad \,\, -U_{3}\Big(\overline{(x_{
\rho(2)}^{\prime\prime})^{+}}^{\rho}\Big(\overline{(x_{
\rho(1)}^{\prime\prime})^{-}}^{\rho}(y^{\prime}\circ A)\Big)\Big)+U_{3}\Big(\overline{(x_{
\rho(2)}^{\prime\prime})^{-}}^{\rho}\Big(\overline{(x_{
\rho(1)}^{\prime\prime})^{-}}^{\rho}(y^{\prime}\circ A)\Big)\Big)\\
&=U_{3}\Big(\Big(\overline{(x_{
\rho(2)}^{\prime\prime})^{+}}^{\rho}\circ \overline{(x_{
\rho(1)}^{\prime\prime})^{+}}^{\rho}+\overline{(x_{
\rho(2)}^{\prime\prime})^{-}}^{\rho}\circ \overline{(x_{
\rho(1)}^{\prime\prime})^{-}}^{\rho}\Big)(y^{\prime}\circ A)\Big)\\
&\quad \,\, -U_{3}\Big(\Big(\overline{(x_{
\rho(2)}^{\prime\prime})^{-}}^{\rho}\circ \overline{(x_{
\rho(1)}^{\prime\prime})^{+}}^{\rho}+\overline{(x_{
\rho(2)}^{\prime\prime})^{+}}^{\rho}\circ \overline{(x_{
\rho(1)}^{\prime\prime})^{-}}^{\rho}\Big)(y^{\prime}\circ A)\Big)\\
&=U_{3}\Bigg(\Bigg(\sum_{P(\xi^{2})=+1}\Big(\overline{(x_{
\rho(2)}^{\prime\prime})^{\xi_{2}}}^{\rho}\circ \overline{(x_{
\rho(1)}^{\prime\prime})^{\xi_{1}}}^{\rho} -\sum_{P(\xi^{2})=-1}\Big(\overline{(x_{
\rho(2)}^{\prime\prime})^{\xi_{2}}}^{\rho}\circ \overline{(x_{
\rho(1)}^{\prime\prime})^{\xi_{1}}}^{\rho} \Bigg)(y^{\prime}\circ A)\Bigg)\\
&\,\, \, \vdots\\
&=U_{m}\Bigg(\Bigg(\sum_{P(\xi^{m-1})=+1}\Big(\overline{(x_{
\rho(m-1)}^{\prime\prime})^{\xi_{m-1}}}^{\rho}\circ\cdots\circ \overline{(x_{
\rho(1)}^{\prime\prime})^{\xi_{1}}}^{\rho}\\
&\quad\,\,  -\sum_{P(\xi^{m-1})=-1}\Big(\overline{(x_{
\rho(m-1)}^{\prime\prime})^{\xi_{m-1}}}^{\rho}\circ\cdots\circ \overline{(x_{
\rho(1)}^{\prime\prime})^{\xi_{1}}}^{\rho} \Bigg)(y^{\prime}\circ A)\Bigg)\\
& =\Bigg[\hspace{-0.05cm}\underbrace{\sum_{P(\xi^{m})=+1}\hspace{-0.3cm}\Big(\overline{(x_{
\rho(m)}^{\prime\prime})^{\xi_{m}}}^{\rho}\circ\cdots\circ \overline{(x_{
\rho(1)}^{\prime\prime})^{\xi_{1}}}^{\rho} \Big)}_{:=T_{1}}-\underbrace{\hspace{-0.1cm}\sum_{P(\xi^{m})=-1}\hspace{-0.3cm}\Big(\overline{(x_{
\rho(m)}^{\prime\prime})^{\xi_{m}}}^{\rho}\circ\cdots\circ \overline{(x_{
\rho(1)}^{\prime\prime})^{\xi_{1}}}^{\rho} \Big)}_{:=T_{2}}\Bigg](y^{\prime}\circ A)\\
&=T_{1}(y^{\prime}\circ A)-T_{2}(y^{\prime}\circ A).
\end{align*}
Taking positive operators $A_{1}, A_{2}\in \mathcal{L}_{r}(E_{1},\ldots,E_{m}; F)$ such that $A=A_{1}-A_{2}$,
\begin{align*}
AB_{n}^{\rho}(A)(x_{1}^{\prime\prime},\ldots,x_{m}^{\prime\prime})(y^{\prime})&=T_{1}(y^{\prime}\circ A)-T_{2}(y^{\prime}\circ A)\\
&=\big(T_{1}(y^{\prime}\circ A_{1})+T_{2}(y^{\prime}\circ A_{2})\big)-\big(T_{1}(y^{\prime}\circ A_{2})+T_{2}(y^{\prime}\circ A_{1})\big).
\end{align*}
For $i=1,2,$ and $y^{\prime}\in F^{\sim}$,
\begin{align*}
T_{1}(y^{\prime}\circ A_{i})&=\sum_{P(\xi^{m})=+1}\Big(\overline{(x_{
\rho(m)}^{\prime\prime})^{\xi_{m}}}^{\rho}\circ\cdots\circ \overline{(x_{
\rho(1)}^{\prime\prime})^{\xi_{1}}}^{\rho} \Big)(y^{\prime}\circ A_{i})\nonumber \\
&=\Bigg[\sum_{P(\xi^{m})=+1} AR_{m}^{\rho}(A_{i})\Big((x_{1}^{\prime\prime})^{\xi_{1}},\ldots,(x_{m}^{\prime\prime})^{\xi_{m}}\Big)\Bigg](y^{\prime}), \label{l1}
\end{align*}
and analogously,
\begin{equation*}\label{l2}
T_{2}(y^{\prime}\circ A_{i})=\Bigg[\sum_{P(\xi^{m})=-1} AR_{m}^{\rho}(A_{i})\Big((x_{1}^{\prime\prime})^{\xi_{1}},\ldots,(x_{m}^{\prime\prime})^{\xi_{m}}\Big)\Bigg](y^{\prime}).
\end{equation*}
Defining, for $i = 1,2,$, the functionals $S_{i}=\hspace{-0.2cm}\displaystyle\sum_{P(\xi^{m})=+1}\hspace{-0.2cm} AR_{m}^{\rho}(A_{i})\Big((x_{1}^{\prime\prime})^{\xi_{1}},\ldots,(x_{m}^{\prime\prime})^{\xi_{m}}\Big) \in F^{\sim\sim}$ and $R_{i}=\hspace{-0.2cm}\displaystyle\sum_{P(\xi^{m})=-1}\hspace{-0.2cm} AR_{m}^{\rho}(A_{i})\Big((x_{1}^{\prime\prime})^{\xi_{1}},\ldots,(x_{m}^{\prime\prime})^{\xi_{m}}\Big)\in F^{\sim\sim}$, it follows that 
\begin{equation*}\label{lu3}
AR_{m}^{\rho}(A)(x_{1}^{\prime\prime},\ldots,x_{m}^{\prime\prime})=(S_{1}+R_{1})-(S_{2}+R_{2}).
\end{equation*}
Since each $(x_{j}^{\prime\prime})^{\xi_{j}}\in E_{j}^{\sim\sim}, j=1,\ldots,n$ is positive, by Proposition \ref{operadores positivos} we know that each $\overline{(x_{
\rho(j)}^{\prime\prime})^{\xi_{j}}}^{\rho}$, is positive as well, therefore the composition $\Big(\overline{(x_{
\rho(m)}^{\prime\prime})^{\xi_{m}}}^{\rho}\circ\cdots\circ \overline{(x_{
\rho(1)}^{\prime\prime})^{\xi_{1}}}^{\rho} \Big)$ is also positive.  And since $A_{1}, A_{2}$ are positive,  $y^{\prime}\circ A_{i}$ is positive for all  $0\leq y^{\prime}\in F^{\sim}$ and $i=1,2$. Hence, $$0\leq \Big(\overline{(x_{
\rho(m)}^{\prime\prime})^{\xi_{m}}}^{\rho}\circ\cdots\circ \overline{(x_{
\rho(1)}^{\prime\prime})^{\xi_{1}}}^{\rho} \Big)(y^{\prime}\circ A_{i})=AR_{m}^{\rho}(A_{i})\Big((x_{1}^{\prime\prime})^{\xi_{1}},\ldots,(x_{m}^{\prime\prime})^{\xi_{m}}\Big)(y^{\prime}),$$
what gives that $S_{i}, R_{i}, i=1,2$, are positive, so $AR_{m}^{\rho}(A)(x_{1}^{\prime\prime},\ldots,x_{m}^{\prime\prime})$ is regular.\\
{\rm (a)} It is easy to see that $AR_{m}^{\rho}(A)$ is $m$-linear, let us prove that it is regular.  
 For $x_{1}^{\prime\prime}\in E_{1}^{\sim\sim},\ldots,x_{m}^{\prime\prime}\in E_{m}^{\sim\sim}$ and 
 $y^{\prime}\in F^{\sim}$,
\begin{align*}
AR_{m}^{\rho}(A)(x_{1}^{\prime\prime},\ldots,x_{m}^{\prime\prime})(y^{\prime})&=\big(\overline{x_{
\rho(m)}^{\prime\prime}}^{\sigma}\circ\cdots\circ\overline
{x_{\rho(1)}^{\prime\prime}}^{\rho}\big)(y^{\prime}\circ A)\\
&=\big(\overline{x_{
\rho(m)}^{\prime\prime}}^{\rho}\circ\cdots\circ\overline
{x_{\rho(1)}^{\prime\prime}}^{\rho}\big)\big((y^{\prime}\circ A_{1})-(y^{\prime}\circ A_{2})\big)\\
&=\big(\overline{x_{
\rho(m)}^{\prime\prime}}^{\rho}\circ\cdots\circ\overline
{x_{\rho(1)}^{\prime\prime}}^{\rho}\big)(y^{\prime}\circ A_{1})-\big(\overline{x_{
\rho(m)}^{\prime\prime}}^{\rho}\circ\cdots\circ\overline
{x_{\rho(1)}^{\prime\prime}}^{\rho}\big)(y^{\prime}\circ A_{2})\\
&=AR_{m}^{\rho}(A_{1})(x_{1}^{\prime\prime},\ldots,x_{m}^{\prime\prime})(y^{\prime})-AR_{m}^{\rho}(A_{2})(x_{1}^{\prime\prime},\ldots,x_{m}^{\prime\prime})(y^{\prime})\\
&=\Big(AR_{m}^{\rho}(A_{1})(x_{1}^{\prime\prime},\ldots,x_{m}^{\prime\prime})-AR_{m}^{\rho}(A_{2})(x_{1}^{\prime\prime},\ldots,x_{m}^{\prime\prime}) \Big)(y^{\prime}),
\end{align*}
proving that 
$AR_{m}^{\rho}(A)=AR_{m}^{\rho}(A_{1})-AR_{m}^{\rho}(A_{2})$. Since $A_{1}$ and $A_{2}$ are positive, $AR_{m}^{\rho}(A_{1})$ and $AR_{m}^{\rho}(A_{2})$ are also positive by what we did above, therefore $AR_{m}^{\rho}(A)$ is regular.\\
{\rm (b)} For all  $x_{1}\in E_{1},\ldots,x_{m}\in E_{m}$ and $y^{\prime}\in F^{\sim}$, applying the definition of $AR_{m}^{\rho}(A)$, Proposition \ref{operadores positivos} and the definition of the maps $J_{E_i}$, we get
\begin{align*}
AR_{m}^{\rho}(A)&(J_{E_{1}}(x_{1}),\ldots,J_{E_{m}}(x_{m}))(y^{\prime})=\big{(}\overline{J_{E_{\rho(m)}}(x_{\rho(m)})}^{\rho}\circ\cdots\circ \overline{J_{E_{\rho(1)}}(x_{\rho(1)})}^{\rho}\big{)}(y^{\prime}\circ A)\\
&=\overline{J_{E_{\rho(m)}}(x_{\rho(m)})}^{\rho}\big(\big{(}\overline
{J_{E_{\rho(m-1)}}(x_{\rho(m-1)})}^{\rho}\circ\cdots\circ \overline{J_{E_{\rho(1)}}(x_{\rho(1)})}^{\rho}\big)(y^{\prime}\circ A)\big{)}\\
&=J_{E_{\rho(m)}}(x_{\rho(m)})\big{(}\big{(}\overline{J_{E_{\rho(m-1)}}(x_{\rho(m-1)})}^{\rho}\circ\cdots\circ \overline{J_{E_{\rho(1)}}(x_{\rho(1)})}^{\rho}\big{)}(y^{\prime}\circ A) \big{)}\\
&=\big{(}\big{(}\overline{J_{E_{\rho(m-1)}}(x_{\rho(m-1)})}^{\rho}\circ\cdots\circ \overline{J_{E_{\rho(1)}}(x_{\rho(1)})}^{\rho}\big{)}(y^{\prime}\circ A)\big{)}(x_{\rho(m)})\\\
&\,\,\,\vdots\\\
&=\overline{J_{E_{\rho(2)}}(x_{\rho(2)})}^{\rho}\big{(}\overline
{J_{E_{\rho(1)}}(x_{\rho(1)})}^{\rho}(y^{\prime}\circ A)\big{)}\left(x_{1},\ldots,\,_{\rho(1)}x,\,_{\rho(2)}x,\ldots,x_{m}\right)\\
&=J_{E_{\rho(2)}}(x_{\rho(2)})\left(\big{(}\overline
{J_{E_{\rho(1)}}(x_{\rho(1)})}^{\rho}(y^{\prime}\circ A)\big{)}\left(x_{1},\ldots,\,_{\rho(1)}x,\,_{\rho(2)}x;\bullet\,;\ldots,x_{m}\right)\right)\\
&=\big{(}\big{(}\overline
{J_{E_{\rho(1)}}(x_{\rho(1)})}^{\rho}(y^{\prime}\circ A)\big{)}\left(x_{1},\ldots,\,_{\rho(1)}x,\,_{\rho(2)}x;\bullet\,;\ldots,x_{m})\right)(x_{\rho(2)})\\
&=\overline{J_{E_{\rho(1)}}(x_{\rho(1)})}^{\rho}(y^{\prime}\circ A)\left(x_{1},\ldots,\,_{\rho(1)}x,\ldots,x_{m}\right)\\
&=J_{E_{\rho(1)}}(x_{\rho(1)})\left((y^{\prime}\circ A)(x_{1},\ldots,\,_{\rho(1)}x;\bullet\,;\ldots,x_{m})\right)\\
&=(y^{\prime}\circ A)\left(x_{1},\ldots,\,_{\rho(1)}x;\bullet\,;\ldots,x_{m}\right)(x_{\rho(1)})=(y^{\prime}\circ A)(x_{1},\ldots,x_{m})\\
&=y^{\prime}(A(x_{1},\ldots,x_{m}))=J_{F}(A(x_{1},\ldots,x_{m}))(y^{\prime})=(J_{F}\circ A)(x_{1},\ldots,x_{m})(y^{\prime}).
\end{align*}

We omit the (easy) proof of (c).

\noindent{\rm (d)} Given $x_{1}\in E_{1},\ldots,x_{m-1}\in E_{m-1}, y^{\prime}\in F^{\sim}, x_{m}^{\prime\prime}\in E_{m}^{\sim\sim}$, since 
$A^{\ast}(y^{\prime},x_{1},\ldots,x_{m-1})=(y^{\prime}\circ A)(x_{1},\ldots,x_{m-1},\bullet)$, we have 
\begin{align*}
A^{\ast\ast}(x_{m}^{\prime\prime},&y^{\prime},x_{1},\ldots,x_{m-2})(x_{m-1})=x_{m}^{\prime\prime}(A^{\ast}(y^{\prime},x_{1},\ldots,x_{m-1}))
=x_{m}^{\prime\prime}((y^{\prime}\circ A)(x_{1},\ldots,x_{m-1},\bullet))\\
&=\overline{x_{m}^{\prime\prime}}^{\theta}(y^{\prime}\circ A)(x_{1},\ldots,x_{m-1})=\overline{x_{m}^{\prime\prime}}^{\theta}(y^{\prime}\circ A)(x_{1},\ldots,x_{m-2},\bullet)(x_{m-1}),
\end{align*}
which gives that $A^{\ast\ast}(x_{m}^{\prime\prime},y^{\prime},x_{1},\ldots,x_{m-2})=\overline{x_{m}^{\prime\prime}}^{\theta}(y^{\prime}\circ A)(x_{1},\ldots,x_{m-2},\bullet)$. And given $x_{1}\in E_{1},\ldots,x_{m-2}\in E_{m-2}, y^{\prime}\in F^{\sim},x_{m-1}^{\prime\prime}\in E_{m-1}^{\sim\sim}, x_{m}^{\prime\prime}\in E_{m}^{\sim\sim}$,
\begin{align*}
A^{\ast\ast\ast}(x_{m-1}^{\prime\prime},& x_{m}^{\prime\prime},y^{\prime},x_{1},
\ldots,x_{m-3})(x_{m-2})=x_{m-1}^{\prime\prime}(A^{\ast\ast}(x_{m}^{\prime\prime},
y^{\prime},x_{1},\ldots,x_{m-2}))\\
&=x_{m-1}^{\prime\prime}(\overline{x_{m}^{\prime\prime}}^{\theta}(y^{\prime}\circ A)(x_{1},\ldots,x_{m-2},\bullet))=\overline{x_{m-1}^{\prime\prime}}^{\theta}(\overline{x_{m}^{\prime\prime}}^{\theta}(y^{\prime}\circ A))(x_{1},\ldots,x_{m-2})\\
&=\overline{x_{m-1}^{\prime\prime}}^{\theta}(\overline{x_{m}^{\prime\prime}}^{\theta}(y^{\prime}\circ A))(x_{1},\ldots,x_{m-3},\bullet)(x_{m-2})\\
&=\big(\overline{x_{m-1}^{\prime\prime}}^{\theta}\circ\overline{x_{m}^{\prime\prime}}^{\theta}\big)(y^{\prime}\circ A)(x_{1},\ldots,x_{m-3},\bullet)(x_{m-2}),
\end{align*}
proving that  $A^{\ast\ast\ast}(x_{m-1}^{\prime\prime},x_{m}^{\prime\prime},y^{\prime},x_{1},\ldots,x_{m-3})=\big(\overline{x_{m-1}^{\prime\prime}}^{\theta}\circ\overline{x_{m}^{\prime\prime}}^{\theta}\big)(y^{\prime}\circ A)(x_{1},\ldots,x_{m-3},\bullet)$. Repeating the procedure $m-2$ times we get that, for all $x_{1}^{\prime\prime}\in E_{1}^{\sim\sim},\ldots,x_{m}^{\prime\prime}\in E_{m}^{\sim\sim}, y^{\prime}\in F^{\sim}$,
$$A^{\ast(m+1)}(x_{1}^{\prime\prime},\ldots,x_{m}^{\prime\prime})(y^{\prime})=\big(\overline{x_{1}^{\prime\prime}}^{\theta}\circ\cdots\circ \overline{x_{m}^{\prime\prime}}^{\theta}\circ\big)(y^{\prime}\circ A)=AR_{m}^{\theta}(A)(x_{1}^{\prime\prime},\ldots,x_{m}^{\prime\prime})(y^{\prime}).$$
\end{proof}

\begin{remark}\label{remmer}\rm In Theorem \ref{propriedadesherdadas}, if $u \in {\cal L}_r(F;G)$, then $AR_{m}^{\rho}(u \circ A) = u'' \circ AR_{m}^{\rho}( A)$. Indeed, for $x_{1}^{\prime\prime}\in E_{1}^{\sim\sim},\ldots,x_{m}^{\prime\prime}\in E_{m}^{\sim\sim}$ and $z' \in G^\sim$,
\begin{align*}AR_{m}^{\rho}(u \circ A)(x_{1}^{\prime\prime},\ldots,x_{m}^{\prime\prime})(z^{\prime})& = \big(\overline{x_{
\rho(m)}^{\prime\prime}}^{\rho}\circ\cdots\circ\overline
{x_{\rho(1)}^{\prime\prime}}^{\rho}\big)(z'\circ u \circ A)\\
&= AR_{m}^{\rho}(A)(x_{1}^{\prime\prime},\ldots,x_{m}^{\prime\prime})(z^{\prime}\circ u)\\
& = AR_{m}^{\rho}(A)(x_{1}^{\prime\prime},\ldots,x_{m}^{\prime\prime})(u'(z^{\prime}))\\
& = (u''\circ AR_{m}^{\rho}(A))(x_{1}^{\prime\prime},\ldots,x_{m}^{\prime\prime})(z^{\prime}).
\end{align*}
\end{remark}

Our next aim is to show that a description of the Arens extensions on Riesz spaces similar to the Davie and Gamelin \cite{gamelin} description of the Aron-Berner extensions  on Banach spaces holds true replacing the weak-star topology with the weak absolute topology $|\sigma|$. This description shall be useful later. For the topology $|\sigma|$ on the order dual of a Riesz space, see  \cite[Chapter 6]{ali}.

\begin{lemma}\label{sigmacontinuo}
Let $E_{1},\ldots,E_{m}, F$ be Riesz spaces,  $A\in\mathcal{L}_{r}(E_{1},\ldots,E_{m};F)$, $\rho\in S_{m}$,  $k\in\{1,\ldots,m\}$, $ x_{\rho(1)}^{\prime\prime}\in E_{\rho(1)}^{\sim\sim},\ldots, x_{\rho(k-1)}^{\prime\prime}\in E_{\rho(k-1)}^{\sim\sim}, x_{\rho(k+1)}\in E_{\rho(k+1)},\ldots,x_{\rho(m)}\in E_{\rho(m)}$. Define
$z_{1}^{\prime\prime},\ldots,z_{m}^{\prime\prime}$ by
\begin{equation}\label{l3}
z_{\rho(j)}^{\prime\prime}=\left\{
\begin{array}{lcr}
x_{\rho(j)}^{\prime\prime} &\text{ if} &j=1,\ldots,k,\\
J_{E_{\rho(j)}}(x_{\rho(j)})& \text{ if} &j=k+1,\ldots,n.
\end{array}
\right.
\end{equation} Then the regular linear operator $$AR_{m}^{\rho}(A)(x_{1}^{\prime\prime},\ldots,\bullet\,,\ldots,x_{m}^{\prime\prime})\colon E_{\rho(k)}^{\sim\sim}\longrightarrow F^{\sim\sim},$$
$$AR_{m}^{\rho}(A)(x_{1}^{\prime\prime},\ldots,\bullet\,,\ldots,x_{m}^{\prime\prime})(x_{\rho(k)}^{\prime\prime})
=AR_{m}(A)(z_{1}^{\prime\prime},\ldots,z_{m}^{\prime\prime}),$$ 
where the dot $\bullet$ is placed at the $\rho(k)$-th coordinate, is
$|\sigma|$-$|\sigma|$ continuous.
\end{lemma}

\begin{proof} 
For $y^{\prime}\in  F^{\sim}$ and  $x_{\rho(k)}^{\prime\prime}\in E_{\rho(k)}^{\sim\sim}$, put $B=y^{\prime}\circ A\in\mathcal{L}_{r}(E_{1},\ldots,E_{m})$ and $w_{\rho(j)}=J_{E_{\rho(j)}}(x_{\rho(j)})$, $j=k+1,\ldots,m$. By the definition of the coordinates of $z_{1}^{\prime\prime},\ldots,z_{m}^{\prime\prime}$,
\begin{align*}
AR_{m}^{\rho}&(A)(x_{1}^{\prime\prime},\ldots,\bullet,\ldots,x_{m}^{\prime\prime})(x_{\rho(k)}^{\prime\prime})(y^{\prime})=AR_{m}(A)(z_{1}^{\prime\prime},\ldots,z_{m}^{\prime\prime})(y^{\prime})\\
&=\big(\overline{w_{\rho(m)}}^{\sigma}\circ \cdots \circ\overline{w_{\rho(k+1)}}^{\sigma}\circ \overline{x_{\rho(k)}^{\prime\prime}}^{\rho}\circ\cdots\circ\overline{x_{\rho(1)}^{\prime\prime}}^{\rho}\big)(B)\\
&=\overline{w_{\rho(m)}}^{\rho}\big(\big(\overline{w_{\rho(m-1)}}^{\sigma}\circ \cdots \circ\overline{w_{\rho(k+1)}}^{\rho}\circ \overline{x_{\rho(k)}^{\prime\prime}}^{\rho}\circ\cdots\circ\overline{x_{\rho(1)}^{\prime\prime}}^{\rho} \big)(B)\big)\\
&=w_{\rho(m)}\big(\big(\overline{w_{\rho(m-1)}}^{\rho}\circ \cdots \overline{w_{\rho(k+1)}}^{\rho}\circ  \overline{x_{\rho(k)}^{\prime\prime}}^{\rho}\circ\cdots\circ\overline{x_{\rho(1)}^{\prime\prime}}^{\rho}\big)(B)\big)\\
&=\big(\overline{w_{\rho(m-1)}}^{\rho}\circ \cdots\circ \overline{w_{\rho(k+1)}}^{\rho}\circ  \overline{x_{\rho(k)}^{\prime\prime}}^{\rho}\circ\cdots\circ\overline{x_{\rho(1)}^{\prime\prime}}^{\sigma}\big)(B)(x_{\rho(m)})\\
&\, \,\,\vdots\\
&=\big( \overline{w_{\rho(k+2)}}^{\rho}\circ\overline{w_{\rho(k+1)}}^{\rho}\circ \overline{x_{\rho(k)}^{\prime\prime}}^{\rho}\circ\cdots\circ\overline{x_{\rho(1)}^{\prime\prime}}^{\rho}\big)(B)(x_{1},\ldots,\,_{\rho(1)}x,\ldots,\,_{\rho(k+2)}x,\ldots,x_{m})\\
&=\overline{w_{\rho(k+2)}}^{\rho}\big(\big(\overline{w_{\rho(k+1)}}^{\rho}\circ \overline{x_{\rho(k)}^{\prime\prime}}^{\rho}\circ\cdots\circ\overline{x_{\rho(1)}^{\prime\prime}}^{\rho} \big)(B)\big)(x_{1},\ldots,\,_{\rho(1)}x,\ldots,\,_{\rho(k+2)}x,\ldots,x_{m})\\
&=w_{\rho(k+2)}\big(\big(\big(\overline{w_{\rho(k+1)}}^{\rho}\circ \overline{x_{\rho(k)}^{\prime\prime}}^{\rho}\circ\cdots\circ\overline{x_{\rho(1)}^{\prime\prime}}^{\rho} \big)(B)\big)(x_{1},\ldots,\,_{\rho(1)}x,\ldots,\,_{\rho(k+2)}x,\bullet,\ldots,x_{m})\big)\\
&=\big(\big(\overline{w_{\rho(k+1)}}^{\rho}\circ \overline{x_{\rho(k)}^{\prime\prime}}^{\rho}\circ\cdots\circ\overline{x_{\rho(1)}^{\prime\prime}}^{\rho} \big)(B)\big)(x_{1},\ldots,\,_{\rho(1)}x,\ldots,\,_{\rho(k+2)}x,\bullet,\ldots,x_{m})(x_{\rho(k+2)})\\
&=\big(\overline{w_{\rho(k+1)}}^{\rho}\circ \overline{x_{\rho(k)}^{\prime\prime}}^{\rho}\circ\cdots\circ\overline{x_{\rho(1)}^{\prime\prime}}^{\rho} \big)(B)(x_{1},\ldots,\,_{\rho(1)}x,\ldots,\,_{\rho(k+1)}x,\ldots,x_{m})\\
&=\overline{w_{\rho(k+1)}}^{\rho}\big(\big( \overline{x_{\rho(k)}^{\prime\prime}}^{\rho}\circ\cdots\circ\overline{x_{\rho(1)}^{\prime\prime}}^{\rho} \big)(B)\big)(x_{1},\ldots,\,_{\rho(1)}x,\ldots,\,_{\rho(k+1)}x,\ldots,x_{m})\\
&=w_{\rho(k+1)}\big(\big(\big( \overline{x_{\rho(k)}^{\prime\prime}}^{\rho}\circ\cdots\circ\overline{x_{\rho(1)}^{\prime\prime}}^{\rho} \big)(B)\big)(x_{1},\ldots,\,_{\rho(1)}x,\ldots,\,_{\rho(k+1)}x,\bullet,\ldots,x_{m})\big)\\
&=\big(\overline{x_{\rho(k)}^{\prime\prime}}^{\rho}\circ\cdots\circ\overline{x_{\rho(1)}^{\prime\prime}}^{\rho} \big)(B)(x_{1},\ldots,\,_{\rho(1)}x,\ldots,\,_{\rho(k+1)}x,\bullet,\ldots,x_{m})(x_{\rho(k+1)})\\
&=\big( \overline{x_{\rho(k)}^{\prime\prime}}^{\rho}\circ\cdots\circ\overline{x_{\rho(1)}^{\prime\prime}}^{\rho} \big)(B)(x_{1},\ldots,\,_{\rho(1)}x,\ldots,\,_{\rho(k)}x,\ldots,x_{m})\\
&= \overline{x_{\rho(k)}^{\prime\prime}}^{\rho}\big(\big( \overline{x_{\rho(k-1)}^{\prime\prime}}^{\rho}\circ\cdots\circ\overline{x_{\rho(1)}^{\prime\prime}}^{\rho} \big)(B)\big)(x_{1},\ldots,\,_{\rho(1)}x,\ldots,\,_{\rho(k)}x,\ldots,x_{m})\\
&=x_{\rho(k)}^{\prime\prime}\big(\underbrace{\big(\big( \overline{x_{\rho(k-1)}^{\prime\prime}}^{\rho}\circ\cdots\circ\overline{x_{\rho(1)}^{\prime\prime}}^{\rho} \big)(B)\big)(x_{1},\ldots,\,_{\rho(1)}x,\ldots,\,_{\rho(k)}x,\bullet,\ldots,x_{m})}_{:=\varphi_{\rho(k)}\in E_{\rho(k)}^{\sim}}\big),
\end{align*}
 that is, $AR_{m}^{\rho}(A)(x_{1}^{\prime\prime},\ldots,\bullet,\ldots,x_{m}^{\prime\prime})(x_{\rho(k)}^{\prime\prime})(y^{\prime})=x_{\rho(k)}^{\prime\prime}(\varphi_{\rho(k)})$ for all $y^{\prime}\in F^{\sim}$ and $x_{\rho(k)}^{\prime\prime}\in E_{\rho(k)}^{\sim\sim}$. Let  $(x_{\alpha_{\rho(k)}}^{\prime\prime})_{\alpha_{\rho(k)}\in \Omega_{\rho(k)}}$ be a net in $E_{\rho(k)}^{\sim\sim}$ such that $x_{\alpha_{\rho(k)}}^{\prime\prime} \xrightarrow{\,\, |\sigma|\,\, } x_{\rho(k)}^{\prime\prime}$. 
 Given $\varphi\in F^{\sim}$, putting $$\psi_{\rho(k)}=\big( \overline{|x_{\rho(k-1)}^{\prime\prime}}|^{\rho}\circ\cdots\circ\overline{|x_{\rho(1)}^{\prime\prime}|}^{\rho} \big)(|\varphi|\circ |A|)(|x_{1}|,\ldots,|_{\rho(1)}x|,\ldots,|_{\rho(k)}x|,\bullet,\ldots,|x_{m}|)\in E_{\rho(k)}^{\sim},$$
the $|\sigma|$-convergence gives
\begin{align*}
|AR_{m}^{\rho}(A)(x_{1}^{\prime\prime},\ldots,\bullet,\ldots,x_{m}^{\prime\prime})&(x_{\alpha_{\rho(k)}}^{\prime\prime})-AR_{m}^{\rho}(A)(x_{1}^{\prime\prime},\ldots,\bullet,\ldots,x_{m}^{\prime\prime})(x_{\rho(k)}^{\prime\prime})|(|\varphi|)\\
&=|AR_{m}^{\rho}(A)(x_{1}^{\prime\prime},\ldots,\bullet,\ldots,x_{m}^{\prime\prime})(x_{\alpha_{\rho(k)}}^{\prime\prime}-x_{\rho(k)}^{\prime\prime})|(|\varphi|)\\
&\leq AR_{m}^{\rho}(|A|)(|x_{1}^{\prime\prime}|,\ldots,\bullet,\ldots,|x_{m}^{\prime\prime}|)(|x_{\alpha_{\rho(k)}}^{\prime\prime}-x_{\rho(k)}^{\prime\prime}|)(|\varphi|)\\
&=|x_{\alpha_{\rho(k)}}^{\prime\prime}-x_{\rho(k)}^{\prime\prime}|(\psi_{\rho(k)})\longrightarrow 0,
\end{align*}
which implies that
$$|AR_{m}^{\rho}(A)(x_{1}^{\prime\prime},\ldots,\bullet,\ldots,x_{m}^{\prime\prime})(x_{\alpha_{\rho(k)}}^{\prime\prime})-AR_{m}^{\rho}(A)(x_{1}^{\prime\prime},\ldots,\bullet,\ldots,x_{m}^{\prime\prime})(x_{\rho(k)}^{\prime\prime})|(|\varphi|)\longrightarrow 0,$$
proving that
$$AR_{m}^{\rho}(A)(x_{1}^{\prime\prime},\ldots,\bullet,\ldots,x_{m}^{\prime\prime})(x_{\alpha_{\rho(k)}}^{\prime\prime})\xrightarrow{\,\, |\sigma|\,\, } AR_{m}^{\rho}(A)(x_{1}^{\prime\prime},\ldots,\bullet,\ldots,x_{m}^{\prime\prime})(x_{\rho(k)}^{\prime\prime}).$$
\end{proof}

\begin{proposition} \label{gamelin}
Let $E_{1},\ldots,E_{m},F$ be Riesz spaces with $E_1, \ldots, E_m$ Archimedean, $\rho\in S_{m}$ and $A \colon E_{1}\times\cdots\times E_{m} \longrightarrow F$ be an $m$-linear regular operator. For $x_k'' \in (E_k^\sim)_n^\sim$, $k = 1, \ldots, m$,
$$AR_{m}^{\rho}(A)(x_1'', \ldots, x_k'') = |\sigma|\mbox{-}\lim_{\alpha_{\rho(1)}}\cdots  |\sigma|\mbox{-}\lim_{\alpha_{\rho(m)}} J_F(A(x_{\alpha_1}, \ldots, x_{\alpha_m})), $$
where $(x_{\alpha_k})_{\alpha_k}$ is a net in $E_k$ such that $J_{E_k}(x_{\alpha_k}) \stackrel{|\sigma|}{\longrightarrow} x_k''$ in $(E_k^\sim)_n^\sim$, $k = 1, \ldots, m$.
\end{proposition}

\begin{proof} 
By \cite[Theorem 2]{grobler} there exist nets $(x_{\alpha_{j}})_{\alpha_{j}\in \Omega_{j}}$ in $E_{j}$ such that $J_{E_{j}}(x_{\alpha_{j}})\xrightarrow{\,\, |\sigma|\,\, } x_{j}^{\prime\prime}$, $j=1,\ldots, k$. Consider $z_{1}^{\prime\prime},\ldots,z_{m}^{\prime\prime}$ as in Lemma \ref{sigmacontinuo}.
In the next computation, the symbol $\stackrel{k=j}{=}$ means that Lemma \ref{sigmacontinuo} is being applied for $k=j$, and the symbol $\stackrel{\ref{l3}, k=j}{=}$ means that (\ref{l3}) is being used for $k=j$: 
\begin{align*}
AR_{m}^{\rho}(A)(x_{1}^{\prime\prime},\ldots,&x_{m}^{\prime\prime})= AR_{m}^{\rho}(A)(x_{1}^{\prime\prime},\ldots,\bullet,\ldots,x_{m}^{\prime\prime})(x_{\sigma(m)}^{\prime\prime})\\
&\hspace{-0.17cm}\stackrel{ k=m}{=} |\sigma|-\lim_{\alpha_{\sigma(m)}}AR_{m}^{\rho}(A)(x_{1}^{\prime\prime},\ldots,\bullet,\ldots,x_{m}^{\prime\prime})(J_{E_{\sigma(m)}}(x_{\alpha_{\sigma(m)}}))\\
&\hspace{-0.48cm}\stackrel{\ref{l3}, k=m-1}{=} |\sigma|-\lim_{\alpha_{\sigma(m)}}AR_{m}^{\rho}(A)(z_{1}^{\prime\prime},\ldots,z_{m}^{\prime\prime})\\
&= |\sigma|-\lim_{\alpha_{\sigma(m)}}AR_{m}^{\rho}(A)(x_{1}^{\prime\prime},\ldots,\bullet,\ldots,x_{m}^{\prime\prime})(x_{\sigma(m-1)}^{\prime\prime})\\
&\hspace{-0.37cm}\stackrel{k=m-1}{=} |\sigma|-\hspace{-0.1cm}\lim_{\alpha_{\sigma(m)}}|\sigma|-\hspace{-0.3cm}\lim_{\alpha_{\sigma(m-1)}}\hspace{-0.2cm} AR_{m}^{\rho}(A)(x_{1}^{\prime\prime},\ldots,\bullet,\ldots,x_{m}^{\prime\prime})(J_{E_{\sigma(m-1)}}(x_{\alpha_{\sigma(m-1)}}))\\
&\hspace{-0.48cm}\stackrel{\ref{l3}, k=m-2}{=} |\sigma|-\lim_{\alpha_{\sigma(m)}}|\sigma|-\lim_{\alpha_{\sigma(m-1)}}AR_{m}^{\rho}(A)(z_{1}^{\prime\prime},\ldots,z_{m}^{\prime\prime})\\
&\hspace{0.2cm}\vdots\\
&\hspace{-0.22cm}\stackrel{\ref{l3}, k=1}{=} |\sigma|-\lim_{\alpha_{\sigma(m)}}\cdots |\sigma|-\lim_{\alpha_{\sigma(2)}}AR_{m}^{\rho}(A)(z_{1}^{\prime\prime},\ldots,z_{m}^{\prime\prime})\\
&= |\sigma|-\lim_{\alpha_{\sigma(m)}}\cdots |\sigma|-\lim_{\alpha_{\sigma(2)}}AR_{m}^{\rho}(A)(x_{1}^{\prime\prime},\ldots,\bullet,\ldots,x_{m}^{\prime\prime})(x_{\sigma(1)}^{\prime\prime})\\
&\hspace{-0.15cm}\stackrel{k=1}{=}|\sigma|-\lim_{\alpha_{\sigma(m)}}\cdots |\sigma|-\lim_{\alpha_{\sigma(1)}}AR_{m}^{\rho}(A)(x_{1}^{\prime\prime},\ldots,\bullet,\ldots,x_{m}^{\prime\prime})(J_{E_{\sigma(1)}}(x_{\alpha_{\sigma(1)}}))\\
&\hspace{-0,033cm}=|\sigma|-\lim_{\alpha_{\sigma(m)}}\cdots |\sigma|-\lim_{\alpha_{\sigma(1)}}AR_{m}^{\rho}(A)(J_{E_{1}}(x_{\alpha_{1}}),\ldots,J_{E_{m}}(x_{\alpha_{m}}))\\
&\hspace{-0.04cm}=|\sigma|-\lim_{\alpha_{\sigma(m)}}\cdots |\sigma|-\lim_{\alpha_{\sigma(1)}}J_{F}(A(x_{\alpha_{1}},\ldots,x_{\alpha_{m}})).
\end{align*}
\end{proof}
\section{Extensions of Riesz multimorphisms}\label{section 3}

Given a Riesz multimorphism $A \in {\cal L}_r(E_1, \ldots, E_m;F)$, from Theorem \ref{propriedadesherdadas}(b) it follows immediately that the restriction of any Arens extension $AR_m^\rho(A)$ of $A$ to $J_{E_1}(E_1) \times \cdots \times J_{E_m}(E_m)$ is a Riesz multimorphism. Moreover, in \cite{boulabiar1} it was proved in the bilinear case that the restriction of the Arens extension $AR_2^\theta(A)$ to $(E_1^\sim)_n^\sim \times  (E_2^\sim)_n^\sim$ is a Riesz multimorphism. Next we show that the formula obtained in Proposition \ref{gamelin} gives a short proof of an  extension of this fact to any Arens extension of $A$ and to any $m$. 

\begin{proposition}\label{fpro} If the Riesz spaces $E_1, \ldots, E_m$ are Archimedean, $F$ is a Riesz space and $A \colon E_{1}\times\cdots\times E_{m} \longrightarrow F$ is a Riesz multimorphism, then, for every $\rho \in S_{m}$, the restriction of $AR^\rho_m(A)$ to $(E_1^\sim)_n^\sim \times \cdots \times (E_m^\sim)_n^\sim$ is a Riesz multimorphism. 
\end{proposition}
\begin{proof}
 Given functionals $x_{1}^{\prime\prime}\in (E_{1}^{\sim})_{n}^{\sim}, \ldots,x_{m}^{\prime\prime}\in (E_{m}^{\sim})_{n}^{\sim}$, let 
 $(x_{\alpha_{j}})_{\alpha_{j}\in \Omega_{j}}$ be nets in $E_{j}$ such that  $J_{E_{j}}(x_{\alpha_{j}})\xrightarrow{\,\, |\sigma| \,\, } x_{j}^{\prime\prime}$, $j = 1, \ldots, m$. The $|\sigma|$-$|\sigma|$ continuity of the lattice operators yields $J_{E_{j}}(|x_{\alpha_{j}}|)\xrightarrow{\,\, |\sigma| \,\, } |x_{j}^{\prime\prime}|$, $j=1,\ldots,m$. Since $A$ is a Riesz multimorphism and $J_F$ is a Riesz homomorphism, applying the $|\sigma|$-$|\sigma|$ continuity of the lattice operators once again and calling on Proposition \ref{gamelin},
\begin{align*}
AR_{m}^{\rho}(A)(|x_{1}^{\prime\prime}|,\ldots,|x_{m}^{\prime\prime}|)&=|\sigma|-\lim_{\alpha_{\rho(m)}}\cdots |\sigma|-\lim_{\alpha_{\rho(1)}}J_{F}(A(|x_{\alpha_{1}}|,\ldots,|x_{\alpha_{m}}|))\\
&=|\sigma|-\lim_{\alpha_{\rho(m)}}\cdots |\sigma|-\lim_{\alpha_{\rho(1)}}|J_{F}(A(x_{\alpha_{1}},\ldots,x_{\alpha_{m}}))|\\
&=\big\|\sigma|-\lim_{\alpha_{\rho(m)}}\cdots |\sigma|-\lim_{\alpha_{\rho(1)}}J_{F}(A(x_{\alpha_{1}},\ldots,x_{\alpha_{m}}))\big|\\
&=|AR_{m}^{\rho}(A)(x_{1}^{\prime\prime},\ldots,x_{m}^{\prime\prime})|.
\end{align*}
\end{proof}

Now we proceed to our main results. Similarly to the notion of finite rank maps between linear spaces (see \cite{mujica}), we say that a map taking values in a Riesz space has {\it finite rank} if the sublattice generated by its range is finite dimensional.


\begin{theorem}\label{AronBernerhr} If $E_1, \ldots, E_{m}$ are Riesz spaces, $F$ is an  Archimedean Riesz space and $A \colon E_1 \times \cdots \times E_{m} \longrightarrow F$ is a finite rank Riesz multimorphism, then all Arens extensions of $A$, $AR_m^\rho(A) \colon E_1^{\sim\sim}\times \cdots \times E_m^{\sim\sim} \longrightarrow F^{\sim\sim}$, $\rho \in S_m$, coincide and are Riesz multimorphisms.
\end{theorem}

\begin{proof} Let $\rho \in S_{m}$. We start with a scalar-valued Riesz multimorphism $B \colon E_1 \times \cdots \times E_{m} \longrightarrow \mathbb{R}$. By Theorem \ref{propriedadesherdadas} we know that $ AR_{m}^{\rho}(B)\colon E_{1}^{\sim\sim}\times
\cdots\times E_{m}^{\sim\sim}\longrightarrow \mathbb{R}$ is a regular $m$-linear form. 
By \cite[Theorem 6]{kus} there are Riesz homomorphisms $\varphi_{i}\colon E_{i}\longrightarrow \mathbb{R}$, $i=1,\ldots,m$, such that $B(x_{1},\ldots,x_{m})=\varphi_{1}(x_{1})\cdots\varphi_{m}(x_{m})$ for all  $x_{1}\in E_{1},\ldots,x_{m}\in E_{m}$. For  $k=1,\ldots,m$, it is plain that
$$B_{k}\colon E_{1}\times\cdots\times \,_{\rho(1)}E\times\cdots\times \,_{\rho(k-1)}E\times\cdots\times E_{m}\longrightarrow\mathbb{R}$$
defined by $ B_{k}(x_{1},\ldots,\,_{\rho(1)}x,\ldots,\,_{\rho(k-1)}x,\ldots,x_{m})=$
$$\varphi_{1}(x_{1})\cdots\,_{\rho(1)}\varphi(x_{{\rho(1)}})
\cdots\,_{\rho(k-1)}\varphi(x_{\rho(k-1)})\cdots\varphi_{m}(x_{m}),$$
is a multilinear form. It is also clear that $B_{k}$ is positive because each  $\varphi_{i}$ is a Riesz homomorphism. In particular, $B_{k}\in\mathcal{L}_{r}( E_{1},\ldots,\,_{\rho(1)}E,\ldots, \,_{\rho(k-1)}E,\ldots, E_{m})$. 
For $k = 1, \ldots, m-1$ and every $x_{\rho(k)}\in E_{\rho(k)}$,
\begin{align*}
B_{k}(x_{1},\ldots,\,_{\rho(1)}x,&\ldots,\,_{\rho(k)}x,\bullet,\ldots,x_{m})(x_{\rho(k)})=B_{k}(x_{1},\ldots,\,_{\rho(1)}x,\ldots,\,_{\rho(k-1)}x,\ldots,x_{m})\\
&=\varphi_{1}(x_{1})\cdots\,_{\rho(1)}\varphi(x_{\rho(1)})\cdots\,_{\rho(k-1)}\varphi(x_{\rho(k-1)})\cdots\varphi_{m}(x_{m})\\
&=\varphi_{1}(x_{1})\cdots\,_{\rho(1)}\varphi(x_{\rho(1)})\cdots\,_{\rho(k)}\varphi(x_{\rho(k)})\cdots\varphi_{m}(x_{m})\varphi_{\rho(k)}(x_{\rho(k)})\\
&=B_{k+1}(x_{1},\ldots,\,_{\rho(1)}x,\ldots,\,_{\rho(k)}x,\ldots,x_{m})\varphi_{\rho(k)}(x_{\rho(k)}),
\end{align*}
that is  $B_{k}(x_{1},\ldots,\,_{\rho(1)}x,\ldots,\,_{\rho(k)}x,\bullet,\ldots,x_{m})=B_{k+1}(x_{1},\ldots,\,_{\rho(1)}x,\ldots,\,_{\rho(k)}x,\ldots,x_{m})\varphi_{\rho(k)}$, $k=1,\ldots,m-1$. And for $k=m$ we have $B_{m}=\varphi_{\rho(m)}$. Given $x_{\rho(k)}^{\prime\prime}\in E_{\rho(k)}^{\sim\sim}$, $k=1,\ldots,m-1$, by Proposition \ref{operadores positivos} we have 
\begin{align*}
\overline{x_{\rho(k)}^{\prime\prime}}^{\rho}(B_{k})(x_{1},&\ldots,\,_{\rho(1)}x,\ldots,\,_{\rho(k)}x,\ldots,x_{m})=x_{\rho(k)}^{\prime\prime}(B_{k}(x_{1},\ldots,\,_{\rho(1)}x,\ldots,\,_{\rho(k)}x;\bullet\,;\ldots,x_{m}))\\
&=x_{\rho(k)}^{\prime\prime}(B_{k+1}(x_{1},\ldots,\,_{\rho(1)}x,\ldots,\,_{\rho(k)}x,\ldots,x_{m})\varphi_{\rho(k)})\\
&= B_{k+1}(x_{1},\ldots,\,_{\rho(1)}x,\ldots,\,_{\rho(k)}x,\ldots,x_{m})x_{\rho(k)}^{\prime\prime}(\varphi_{\rho(k)})).
\end{align*}
It follows that  $\overline{x_{\rho(k)}^{\prime\prime}}^{\rho}(B_{k})=x_{\rho(k)}^{\prime\prime}(\varphi_{\rho(k)}))B_{k+1}$ for $k=1,\ldots,m-1$. And for $k=m$ we have $\overline{x_{\rho(m)}^{\prime\prime}}^{\rho}(B_{m})=\overline{x_{\rho(m)}^{\prime\prime}}^{\rho}(\varphi_{\rho(m)})=x_{\rho(m)}^{\prime\prime}(\varphi_{\rho(m)})$. Using that $B=B_{1}$, for all $x_{1}^{\prime\prime}\in E_{1}^{\sim\sim},\ldots,x_{m}^{\prime\prime}\in E_{m}^{\sim\sim}$,
\begin{align*}
AR_{m}^{\rho}(B)(x_{1}^{\prime\prime},\ldots,x_{m}^{\prime\prime})&=\big(\overline{x_{\rho(m)}^{\prime\prime}}^{\rho}\circ\cdots\circ \overline{x_{\rho(1)}^{\prime\prime}}^{\rho}\big)(B)=\big(\overline{x_{\rho(m)}^{\prime\prime}}^{\rho}\circ\cdots\circ \overline{x_{\rho(1)}^{\prime\prime}}^{\rho}\big)(B_{1})\\
&=\overline{x_{\rho(m)}^{\prime\prime}}^{\rho}\big(\cdots\big( \overline{x_{\rho(2)}^{\prime\prime}}^{\rho}\big(\overline{x_{\rho(1)}^{\prime\prime}}^{\rho}(B_{1})\big)\big)\cdots \big)\\
&=\overline{x_{\rho(m)}^{\prime\prime}}^{\rho}\big(\cdots\big( \overline{x_{\rho(2)}^{\prime\prime}}^{\rho}\big(x_{\rho(1)}^{\prime\prime}(\varphi_{\rho(1)})B_{2}\big)\big)\cdots \big)\\
&=x_{\rho(1)}^{\prime\prime}(\varphi_{\rho(1)})\overline{x_{\rho(m)}^{\prime\prime}}^{\rho}\big(\cdots\big( \overline{x_{\rho(3)}^{\prime\prime}}^{\rho} \big( \overline{x_{\rho(2)}^{\prime\prime}}^{\rho}(B_{2})\big) \big)\cdots\big)\\
&=x_{\rho(1)}^{\prime\prime}(\varphi_{\rho(1)})\overline{x_{\rho(m)}^{\prime\prime}}^{\rho}\big(\cdots\big( \overline{x_{\rho(3)}^{\prime\prime}}^{\rho} \big( x_{\rho(1)}^{\prime\prime}(\varphi_{\rho(2)})B_{3}\big) \big)\cdots\big)\\
&=x_{\rho(1)}^{\prime\prime}(\varphi_{\rho(1)})x_{\rho(2)}^{\prime\prime}(\varphi_{\rho(2)})\overline{x_{\rho(m)}^{\prime\prime}}^{\rho}\big(\cdots\big( \overline{x_{\rho(3)}^{\prime\prime}}^{\rho} ( B_{3}) \big)\cdots\big)\\
&\,\,\,\vdots\\
&=x_{\rho(1)}^{\prime\prime}(\varphi_{\rho(1)})\cdots x_{\rho(m-2)}^{\prime\prime}(\varphi_{\rho(m-2)})\overline{x_{\rho(m)}^{\prime\prime}}^{\rho}\big( \overline{x_{\rho(m-1)}^{\prime\prime}}^{\rho} ( B_{m-1}) \big)\\
&=x_{\rho(1)}^{\prime\prime}(\varphi_{\rho(1)})\cdots x_{\rho(m-2)}^{\prime\prime}(\varphi_{\rho(m-2)})\overline{x_{\rho(m)}^{\prime\prime}}^{\rho}\big( x_{\rho(m-1)}^{\prime\prime}(\varphi_{\rho(m-1)})B_{m} \big)\\
&=x_{\rho(1)}^{\prime\prime}(\varphi_{\rho(1)})\cdots x_{\rho(m-2)}^{\prime\prime}(\varphi_{\rho(m-2)})x_{\rho(m-1)}^{\prime\prime}(\varphi_{\rho(m-1)})\overline{x_{\rho(m)}^{\prime\prime}}^{\rho}\big( \varphi_{\rho(m)} \big)\\
&=x_{\rho(1)}^{\prime\prime}(\varphi_{\rho(1)})\cdots x_{\rho(m-2)}^{\prime\prime}(\varphi_{\rho(m-2)})x_{\rho(m-1)}^{\prime\prime}(\varphi_{\rho(m-1)})x_{\rho(m)}^{\prime\prime}(\varphi_{\rho(m)}),
\end{align*}
that is, 
$$AR_{m}^{\rho}(B)(x_{1}^{\prime\prime},\ldots,x_{m}^{\prime\prime})=x_{1}^{\prime\prime}(\varphi_{1})\cdots x_{m}^{\prime\prime}(\varphi_{m})=\varphi_{1}^{\prime\prime}(x_{1}^{\prime\prime})\cdots \varphi_{m}^{\prime\prime}(x_{m}^{\prime\prime}),$$
where $\varphi_{i}^{\prime\prime}$ is the second adjoint of the Riesz homomorphism $\varphi_{i}$, $i=1,\ldots,m$. 
So, the extension does not depend on $\rho$ and, since $\varphi_{1}^{\prime\prime}, \ldots, \varphi_m^{''}$ are Riesz homomorphisms \cite[Theorems 2.19 and 2.20]{positiveoperators}, it follows immediately that $AR_{m}^{\rho}(B)$ is a Riesz multimorphism.

Now, let $n \in \mathbb{N}$ be given, consider $\mathbb{R}^n$ with the coordinatewise order and let  $C \colon E_1 \times \cdots \times E_{m} \longrightarrow \mathbb{R}^n$ be a  a Riesz multimorphism. By Theorem \ref{propriedadesherdadas} we know that $ AR_{m}^{\rho}(C)\colon E_{1}^{\sim\sim}\times
\cdots\times E_{m}^{\sim\sim}\longrightarrow (\mathbb{R}^n)^{\sim\sim} = \mathbb{R}^n$ is a regular $m$-linear form. Taking the canonical projections $\pi_j \colon \mathbb{R}^n \longrightarrow \mathbb{R}$, $j = 1, \ldots, n$, which are Riesz multimorphisms, each $\pi_j \circ C$ is a scalar-valued Riesz multimorphism, so the first part of the proof gives that $ AR_{m}^{\rho}(\pi_j \circ C)\colon E_{1}^{\sim\sim}\times
\cdots\times E_{m}^{\sim\sim}\longrightarrow \mathbb{R}$ is a Riesz multimorphism that does not depend on $\rho$. By Remark \ref{remmer},
$$ \pi_j \circ  AR_{m}^{\rho}(C) =  \pi_j'' \circ  AR_{m}^{\rho}(C)= AR_{m}^{\rho}(\pi_j \circ C)$$
is a Riesz multimorphism for every $j= 1, \ldots, n$. From
$$AR_{m}^{\rho}(C)(x_1'', \ldots, x_m'') = ((\pi_j \circ AR_{m}^{\rho}(C)(x_1'', \ldots, x_m''))_{j=1}^n $$
it follows that $AR_{m}^{\rho}(C)$ does not depend on $\rho$ and is a Riesz multimorphism because we have the coordinatewise order in $\mathbb{R}^n$.

Finally, given an Archimedean Riesz space $F$  and a finite rank Riesz multimorphism $A \colon E_1 \times \cdots \times E_{m} \longrightarrow F$, call $G$ the finite-dimensional sublattice of $F$ generated by the range of $A$, denote by $i \colon G \longrightarrow F$ the inclusion operator and note that the astriction of $A$ to $G$, $A_1 \colon E_1 \times \cdots \times E_{m} \longrightarrow G$, is a Riesz multimorphism. Since $G$ is Archimedean, letting $n \in \mathbb{N}$ denote its dimension,  by \cite[Corollary 1, p.\,70]{schaefer} there is a Riesz isomorphism $I \colon G \longrightarrow \mathbb{R}^n$, where $\mathbb{R}^n$ has the coordinatewise order. Noting that $I \circ A_1$ is a a $\mathbb{R}^n$-valued Riesz homomorphism, the second part of the proof gives that $AR_{m}^{\rho}(I\circ A_1)$ is a Riesz multimorphism not depending on $\rho$. By  Remark \ref{remmer},
$$AR_{m}^{\rho}(A) = AR_{m}^{\rho}(i \circ I^{-1} \circ I \circ A_1) = i'' \circ (I^{-1})''\circ AR_{m}^{\rho}(I \circ A_1)$$
is a Riesz multimorphism, not depending on $\rho$, because $i''$ and $(I^{-1})''$ are Riesz homomorphisms.
\end{proof}

\begin{remark}\rm Let $E_{1},\ldots,E_{m}$ be Riesz spaces, $(\mathcal{A}, \ast)$ be an $f$-algebra and $A\colon E_{1}\times\cdots\times E_{m}\longrightarrow \mathcal{A}$ be a multiplicative Riesz multimorphism, that is, $A(x_{1},\ldots ,x_{m})=T_{1}(x_{1})\ast \cdots\ast T_{m}(x_{m})$ for all $x_{i}\in E_{i}, i=1,\ldots,m$, where each $T_{i}\colon E_{i}\longrightarrow \mathcal{A}$ is a  Riesz homomorphism. 
The first part of the proof above can be adapted to show that all Arens extensions of $A$ are Riesz multimorphisms. Actually, the reasoning shows that, for every $ \rho\in S_{m}$  and all   $x_{i}^{\prime\prime}\in E_{i}^{\sim\sim}, i=1,\ldots,m$,
$$AR_{m}^{\rho}(A)(x_{1}^{\prime\prime},\ldots,x_{m}^{\prime\prime})=(x_{\rho(m)}^{\prime\prime}\circ T_{\rho(m)}^{\prime})\odot\cdots\odot (x_{\rho(1)}^{\prime\prime}\circ T_{\rho(1)}^{\prime}),$$
where $\odot$ is the Arens product that makes $\mathcal{A}^{\sim\sim}$ an $f$-algebra.

This shows, in particular, that  all Arens extensions of any Riesz multimorphism taking values in a universally complete Riesz space with a weak order unit are Riesz multimorphisms.
\end{remark}

The following partial results for arbitrary Riesz multimorphisms will be helpful later.

\begin{proposition} \label{corohmcve1}
Let $E_1, \ldots, E_m, F$ be Riesz spaces, $A \in {\cal L}_r(E_1, \ldots, E_m;F)$ be a Riesz multimorphism and $\rho \in S_m$. For every Riesz homomorphism $y' \in F^{\sim}$, $J_{F^{\sim}}(y') \circ AR_m^{\rho}(A)$ is a Riesz multimorphism and
$$|AR_{m}^{\rho}(A)(x_{1}^{\prime\prime},\ldots,x_{m}^{\prime\prime})|(y^{\prime})=AR_{m}^{\rho}(A)(|x_{1}^{\prime\prime}|,\ldots,|x_{m}^{\prime\prime}|)(y^{\prime})$$
for all $x_{1}^{\prime\prime}\in E_{1}^{\sim\sim},\ldots,x_{m}^{\prime\prime}\in E_{m}^{\sim\sim}$.
\end{proposition}

\begin{proof} Let a Riesz homomorphism $y' \in F^{\sim}$ be given. Then $y' \circ A \in {\cal L}_r(E_1, \ldots, E_m)$ is a Riesz multimorphism, so $AR^\rho_m(y' \circ A)\colon E_{1}^{\sim\sim}\times\cdots\times E_{m}^{\sim\sim}\longrightarrow \mathbb{R}$ is a Riesz multimorphism by Theorem \ref{AronBernerhr}. To prove the first assertion it is enough to check that $J_{F^{\sim}}(y') \circ AR_m^{\rho}(A) = AR^\rho_m(y' \circ A)$: 
for all $x_{1}^{\prime\prime}\in E_{1}^{\sim\sim}, \ldots,x_{m}^{\prime\prime}\in E_{m}^{\sim\sim}$,
\begin{align*}
(J_{F^{\sim}}(y') \circ AR_m^{\rho}(A))(x_{1}^{\prime\prime},\ldots,x_{m}^{\prime\prime})&=J_{F^{\sim}}(y')(AR_m^{\rho}(A)(x_{1}^{\prime\prime},\ldots,x_{m}^{\prime\prime}))\\
&=AR_m^{\rho}(A)(x_{1}^{\prime\prime},\ldots,x_{m}^{\prime\prime})(y')\\
&=\big(\overline{x_{\rho(m)}^{\prime\prime}}^{\rho}\circ\cdots\circ \overline{x_{\rho(1)}^{\prime\prime}}^{\rho} \big)(y'\circ A)\\
&=AR_m^{\rho}(y'\circ A)(x_{1}^{\prime\prime},\ldots,x_{m}^{\prime\prime}).
\end{align*}

Once we have just proved that $J_{F^{\sim}}(y^{\prime})\circ AR_{m}^{\rho}(A)$ is a Riesz multimorphism and using that $AR_m^{\rho}(A) \geq 0$  (Theorem \ref{propriedadesherdadas}) and $y' \geq 0$,  for all  $x_{1}^{\prime\prime}\in E_{1}^{\sim\sim},\ldots,x_{m}^{\prime\prime}\in E_{m}^{\sim\sim}$ we have
\begin{align*}
AR_m^{\rho}(A)(|x_{1}^{\prime\prime}|,\ldots,|x_{m}^{\prime\prime}|)(y^{\prime})&=(J_{F^{\sim}}(y^{\prime}) \circ AR_m^{\rho}(A))(|x_{1}^{\prime\prime}|,\ldots,|x_{m}^{\prime\prime}|)\\
&=|(J_{F^{\sim}}(y^{\prime}) \circ AR_m^{\rho}(A))(x_{1}^{\prime\prime},\ldots,x_{m}^{\prime\prime})|\\
&=|AR_m^{\rho}(A)(x_{1}^{\prime\prime},\ldots,x_{m}^{\prime\prime})(y^{\prime})|\\
&\leq |AR_m^{\rho}(A)(x_{1}^{\prime\prime},\ldots,x_{m}^{\prime\prime})|(y^{\prime})\\
&\leq AR_m^{\rho}(A)(|x_{1}^{\prime\prime}|,\ldots,|x_{m}^{\prime\prime}|)(y^{\prime}),
\end{align*}
from which the second assertion follows.
\end{proof}

Henceforth we present our results in the environment of Banach lattices. We start with the following immediate consequences of our previous results. 

\begin{proposition} Let $E_1, \ldots, E_m,F$ be Banach lattices.\\
{\rm (a)} If $A\in\mathcal{L}_{r}(E_{1},\ldots,E_{m};F)$ is a finite rank Riesz multimorphism, then the Aron-Berner extensions of $A$, $AB_{m}^{\rho}(A)\colon E_{1}^{**}\times
\cdots\times E_{m}^{**}\longrightarrow F^{**}$, $\rho \in S_m$, coincide and are Riesz multimorphisms.\\
{\rm (b)} If $E_1^*, \ldots, E_{m}^*$ have order continuous norms and $A\in\mathcal{L}_{r}(E_{1},\ldots,E_{m};F)$ is a Riesz multimorphism, then, for every $\rho \in S_m$,  the Aron-Berner extension $AB_{m}^{\rho}(A)\colon E_{1}^{\ast\ast}\times
\cdots\times E_{m}^{\ast\ast}\longrightarrow F^{\ast\ast}$ is a Riesz multimorphism.
\end{proposition}

\begin{proof} For a Banach lattice $E$, $E^\sim$ is a Banach lattice and $E^{\sim}=E^*$, thus $E^{\sim\sim}=E^{**}$  \cite[Theorem 4.1 and Corollary 4.5]{positiveoperators}. Thus, (a) follows from Theorem \ref{AronBernerhr}. And   $(E^{\sim})_{n}^{\sim}=E^{\ast\ast}$ whenever $E^*$ has order continuous norm \cite[Theorem 2.4.2]{nieberg}. Thus, (b) follows from Proposition \ref{fpro}. 
\end{proof}

The next results concern Aron-Berner extensions of vector-valued Riesz homomorphisms defined on arbitrary Banach lattices.
%
%
%
%
%
%

\begin{proposition} \label{corohmcve2}
Let $E_1, \ldots, E_m, F$ be Banach lattices, $A \in {\cal L}_r(E_1, \ldots, E_m;F)$ be a Riesz multimorphism and $\rho \in S_m$. Then:\\
{\rm (a)} $y^{***} \circ AB_m^{\rho}(A)$ is a Riesz multimorphism for every $w^*$-continuous Riesz homomorphism  $y^{***} \in F^{***}$.\\
{\rm (b)} It holds
$$|AB_{m}^{\rho}(A)(x_{1}^{**},\ldots,x_{m}^{**})|(y^*)=AB_{m}^{\rho}(A)(|x_{1}^{**}|,\ldots,|x_{m}^{**}|)(y^{*}),$$
for all $x_{1}^{**}\in E_{1}^{**},\ldots,x_{m}^{**}\in E_{m}^{**}$ and any $y^* \in \overline{{\rm span}}\{\varphi\in F^{\ast}:\varphi \text{ is a Riesz homomorphism}\}.$
\end{proposition}

\begin{proof}
{\rm (a)} Given a $\omega^{\ast}$-continuous Riesz  homomorphisms  $y^{\ast\ast\ast}\in F^{\ast\ast\ast}$, take  $y^{\ast}\in F^{\ast}$ such that $y^{\ast\ast\ast}=J_{F^{\ast}}(y^{\ast})$. For every $y\in F$,
$$(y^{\ast\ast\ast}\circ J_{F})(y)=y^{\ast\ast\ast}(J_{F}(y))=J_{F^{\ast}}(y^{\ast})(J_{F}(y))=J_{F}(y)(y^{\ast})=y^{\ast}(y),$$
that is, $y^{\ast}=y^{\ast\ast\ast}\circ J_{F}$. As the composition of two Riesz homomorphisms, $y^{\ast}$ is a Riesz homomorphism as well. Corollary \ref{corohmcve1} gives that $y^{***} \circ AB_m^{\rho}(A) = J_{F^{\ast}}(y^{\ast})\circ AB_{m}^{\rho}(A)$ is a Riesz multimorphism.\\ 
{\rm (b)} 
Let $x_{1}^{\ast\ast}\in E_{1}^{\ast\ast},\ldots,x_{m}^{\ast\ast}\in E_{m}^{\ast\ast}$ and  $y^{\ast}\in \overline{{\rm span}}\{\varphi\in F^{\ast}:\varphi \text{ is a Riesz homomorphism}\}$ be given. Take a sequence $(y_{n}^{\ast})_{n=1}^{\infty}$ in ${\rm span}\{\varphi\in F^{\ast}:\varphi \text{ is a Riesz homomorphism}\}$  such that $y_{n}^{\ast}\longrightarrow y^{\ast}$. For each  $n\in\mathbb{N}$ there are $k_n \in \mathbb{N}$, Riesz homomorphisms $\varphi^1_n,\ldots,\varphi_n^{k_n}\in F^{\ast}$ and scalars $\alpha_n^1,\ldots,\alpha_n^{k_n}$ such that  $y_{n}^{\ast}=\sum\limits_{j=1}^{k_n}\alpha_n^{j}\varphi_n^j$. The continuity of   $|AB_{m}^{\rho}(A)(x_{1}^{\ast\ast},\ldots,x_{m}^{\ast\ast})|$ and  $AB_{m}^{\rho}(A)(|x_{1}^{\ast\ast}|,\ldots,|x_{m}^{\ast\ast}|)$ and Corollary \ref{corohmcve1} give
\begin{align*}
|AB_{m}^{\rho}(A)(x_{1}^{\ast\ast},\ldots,x_{m}^{\ast\ast})|(y^*)&=\lim_{n\rightarrow \infty} |AB_{m}^{\rho}(A)(x_{1}^{\ast\ast},\ldots,x_{m}^{\ast\ast})|(y_{n}^{\ast})\\
&=\lim_{n\rightarrow \infty} \displaystyle\sum_{j=1}^{k_n} \alpha_n^{j} |AB_{m}^{\rho}(A)(x_{1}^{\ast\ast},\ldots,x_{m}^{\ast\ast})|(\varphi_n^{j})\\
&=\lim_{n\rightarrow \infty} \displaystyle\sum_{j=1}^{k_n} \alpha_n^{j} AB_{m}^{\rho}(A)(|x_{1}^{\ast\ast}|,\ldots,|x_{m}^{\ast\ast}|)(\varphi_n^{j})\\
&=\lim_{n\rightarrow \infty} AB_{m}^{\rho}(A)(|x_{1}^{\ast\ast}|,\ldots,|x_{m}^{\ast\ast}|)(y_{n}^{\ast})\\
&=AB_{m}^{\rho}(A)(|x_{1}^{\ast\ast}|,\ldots,|x_{m}^{\ast\ast}|)(y^{\ast}).
\end{align*}
\end{proof}

From now on, for a Banach lattice $F$ the expression {\it all Aron-Berner extensions of any $F$-valued Riesz multimorphism are Riesz multimorphisms} means that, regardless of the natural number $m$ and the Banach lattices $E_1, \ldots, E_m$, all Aron-Berner extensions of any Riesz multimorphism from  $E_1 \times \cdots \times E_m$ to $F$ are Riesz multimorphisms on $E_1^{**} \times \cdots \times E_m^{**}$.

 Let $F$ be a Banach lattice such that $F^*$ has a Schauder basis formed by Riesz homomorphisms. Then $F^* = \overline{{\rm span}}\{\varphi\in F^{\ast}:\varphi \text{ is a Riesz homomorphism}\}$, so the next result follows immediately from Propositon \ref{corohmcve2}(b).

\begin{corollary}\label{corohmcve3} Let $F$ be a Banach lattice such that $F^*$ has a Schauder basis formed by Riesz homomorphisms. Then all Aron-Berner extensions of any $F$-valued Riesz multimorphism are Riesz multimorphisms.
%
\end{corollary}

 Before giving concrete examples, let us see a simple result in the realm of Riesz spaces.

\begin{proposition}\label{nnpro} If $G$ is a projection band in the Riesz space $F$ and all Arens extensions of any $F$-valued Riesz multimorphism are Riesz multimorphisms, then the same holds for $G$-valued Riesz multimorphisms. \end{proposition}
\begin{proof} Let $E_1, \ldots, E_m$ be Riesz spaces,  $A \in {\cal L}_r(E_1, \ldots, E_m;G)$ be a Riesz multimorphism and $\rho \in S_m$. Denoting by $i \colon G \longrightarrow F$ the inclusion operator and by $\pi \colon F \longrightarrow G$ the corresponding band projection, which is a Riesz homomorphism, it follows that
$$AR_m^\rho(A) = AR_m^\rho(\pi \circ i \circ A) = \pi'' \circ AR_m^\rho(i \circ A) $$
is a Riesz multimorphism since $AR_m^\rho(i \circ A)$ is a Riesz multimorphism by assumption and $\pi''$ is a Riesz homomorphism. 
\end{proof}

\begin{example}\label{ffex}\rm (a) The canonical unit vectors $(e_j)_{j=1}^\infty$ is a Schauder basis formed by Riesz homomorphisms in
  $\ell_p$, $1 \leq p < \infty$.
So, Corollary \ref{corohmcve3} applies to $F = c_0$ and $F = \ell_p$, $1 < p < \infty$.

Henceforth in this example, whenever we say that a Banach space $F$ has an 1-uncondi-tional Schauder basis, $F$ is regarded as a Banach lattice with the order given by the basis. \\
(b) Let $F$ be a Banach space with an 1-unconditional Schauder basis $(x_j)_{j=1}^\infty$ not containing a copy of $\ell_1$. The basis is shrinking by \cite[Proposition 1.b.1]{ltI}, so the biorthogonal functionals $(x_j^*)_{j=1}^\infty$ associated to $(x_j)_{j=1}^\infty$ form a Schauder basis of $F^*$ by \cite[Proposition 1.b.1]{ltI}. It is easy to check that each $x_j^*$ is a Riesz homomorphism, so Corollary \ref{corohmcve3} applies to $F$. \\
(c) By (b), Corollary \ref{corohmcve3} applies to every reflexive Banach space with an 1-unconditional basis. Just to give reflexive examples different from $\ell_p$, $1 < p < \infty$, note that Tsirelson's original space $T^*$ \cite{tsirelson} and its dual $T$ are reflexive spaces with 1-unconditional bases \cite[Theorem I.8 and Notes and Remarks p.\,16)]{shura}. So, Corollary \ref{corohmcve3} applies to $F = T^*$ and $F = T$.\\
(d) Just to illustrate, let us give two nonreflexive examples different from $c_0$. Schreier's space $S$ \cite{schreier} and the predual $d_*(w,1)$ of the Lorenz sequence space $d(w,1)$ \cite[4.e]{ltI} are nonreflexive Banach spaces with  1-unconditional bases not containing a copy of $\ell_1$ (for $S$ see \cite[Corollary 0.8, Proposition 0.4, Theorem 0.5]{shura} and for $d_*(w,1)$ see \cite[p.\,1202]{eloi}, \cite[p.\,1643]{albiac} and \cite[p.\,19]{ltI}). By (b), Corollary \ref{corohmcve3} applies to $F = S$ and $F = d_*(w,1)$.\\
(e) By Proposition \ref{nnpro}, Corollary \ref{corohmcve3} applies to every Banach lattice $F$ that is a projection band in any of the Banach lattices listed above.
\end{example}

Our last purpose is to enlarge substantially the class of Banach lattices $F$ for which all Aron-Berner extensions of any $F$-valued Riesz multimorphism are Riesz multimorphisms. To do so, recall that, given a sequence
$(F_{n})_{n=1}^{\infty}$ of Banach lattices and $1\leq p<\infty$, 
$$\left(\oplus_n F_{n}\right)_{p}=\left\{(x_{n})_{n=1}^{\infty}: x_{n}\in F_{n} \text{ and } \|(x_{n})_{n=1}^{\infty}\|_{p}=\Big(\sum_{n=1}^{\infty}\|x_{n}\|^{p}\Big)^{\frac{1}{p}}<\infty\right\} \mbox{ and}$$
$$\left(\oplus_n F_{n}\right)_{0}=\{(x_{n})_{n=1}^{\infty}: x_{n}\in F_{n} \text{ and } \|x_{n}\|\longrightarrow 0\}, \|(x_{n})_{n=1}^{\infty}\| = \sup_n\|x_n\|,$$
are Banach lattices with the coordinatewise order. If $F = F_n$ for every $n$, it is usual to write $\ell_p(F)$ and $c_0(F)$ instead of $\left(\oplus_n F\right)_{p}$ and $\left(\oplus_n F\right)_{0}$.

\begin{proposition}\label{llpro} Let $1 < p < \infty$ and let $(F_n)_{n=1}^\infty$ be a sequence of Banach lattices such that, for every $n$, all Aron-Berner extensions of any $F_n$-valued Riesz multimorphism are Riesz multimorphisms. Then all Aron-Berner extensions of any $\left(\oplus_n F_n\right)_0$-valued and any $\left(\oplus_n F_n\right)_p$-valued Riesz multimorphism are Riesz multimorphisms.
\end{proposition}

\begin{proof} Let $m \in \mathbb{N}$, $\rho \in S_m$ and let $E_1, \ldots, E_m$ be Banach lattices.

Given a Riesz multimorphism $A\in\mathcal{L}_{r}(E_{1},\ldots,E_{m};(\oplus_n F_{n})_{p})$, for each $k\in\mathbb{N}$ the projection $\pi_{k}\colon (\oplus_n F_{n})_p\longrightarrow F_{k},\,  \pi_{k}((y_{n})_{n=1}^{\infty})=y_{k}$, is a Riesz homomorphism, hence $\pi_{k}\circ A\in\mathcal{L}_{r}(E_{1},\ldots,E_{m};F_{k})$ is a Riesz multimorphism. By assumption,
\begin{equation}\label{mult2}
AB_{m}^{\rho}(\pi_{k}\circ A)\in \mathcal{L}_{r}(E^{\ast\ast}_{1},\ldots,E^{\ast\ast}_{m};F_{k}^{\ast\ast}) \mbox { is a Riesz multimorphism for every } k.
\end{equation}
Taking $q$ such that $\frac{1}{p}+\frac{1}{q}=1$, it is well know that the maps
\begin{align*}
\psi_{1}\colon (\oplus_n F_{n}^{\ast})_{q}\longrightarrow (\oplus_n F_{n})_{p}^{\ast}~,~\psi_{1}((y_{n}^{\ast})_{n=1}^{\infty})((y_{n})_{n=1}^{\infty})=\sum_{n=1}^{\infty}y_{n}^{\ast}(y_{n}),\\
\psi_{2}\colon (\oplus_n F_{n}^{\ast\ast})_{p}\longrightarrow (\oplus_n F_{n}^{\ast})_{q}^{\ast}~,~ \psi_{2}((y_{n}^{\ast\ast})_{n=1}^{\infty})((y_{n}^{\ast})_{n=1}^{\infty})=\sum_{n=1}^{\infty}y_{n}^{\ast\ast}(y_{n}^{\ast}),
\end{align*}
are lattice isomorphisms, that is, Banach space isomorphisms + Riesz homomorphisms (see \cite[Theorem 4.6 and its proof]{positiveoperators}). So, $\psi_3 :=(\psi_1^{-1})^* \colon (\oplus_n F_{n}^{\ast})^{\ast}_{q}\longrightarrow (\oplus_n F_{n})_{p}^{\ast\ast}$ and $\psi:=\psi_{2}^{-1}\circ \psi_{3}^{-1}\colon  (\oplus_n F_{n})_{p}^{\ast\ast}\longrightarrow (\oplus_n F_{n}^{\ast\ast})_{p}$ are Banach space isomorphisms. It is clear that $\psi_3^{-1} = \psi_1^* \colon (\oplus_n F_{n})_{p}^{\ast\ast}\longrightarrow (\oplus_n F_{n}^{\ast})^{\ast}_{q}$. For each $k \in \mathbb{N}$, by $i_k \colon F_k^* \longrightarrow (\oplus_n F_{n}^{\ast})_{q}$ we denote the canonical embedding, $i_k(y_k^*) = (0, \ldots, 0, y_k^*,0, \ldots)$.
It will be useful to bear in mind that the inverse
%
%
%
of $\psi_{2}$ is given by
$$\psi_{2}^{-1}\colon (\oplus_n F_{n}^{\ast})_{q}^{\ast}\longrightarrow (\oplus_n F_{n}^{\ast\ast})_{p}~,~ \psi_{2}^{-1}(y^{\ast\ast})=(y^{\ast\ast} \circ i_n)_{n=1}^{\infty}.$$
   For all $y^{\ast}_{k}\in F_{k}^{\ast}$ and $(y_{n})_{n=1}^{\infty}\in (\oplus_n F_{n})_{p}$,
$$(y^{\ast}_{k}\circ \pi_{k})((y_{n})_{n=1}^{\infty})=y_{k}^{\ast}(\pi_{k}((y_{n})_{n=1}^{\infty}))=y_{k}^{\ast}(y_{k})=\psi_{1}(0,\ldots,0,y_{k}^{\ast},0,\ldots)((y_{n})_{n=1}^{\infty}),$$
showing that $y^{\ast}_{k}\circ \pi_{k}=\psi_{1}(0,\ldots,0,y_{k}^{\ast},0,\ldots).$ So, for $x_{j}^{\ast\ast}\in E_{j}^{\ast\ast}, j=1,\ldots,m$, $k\in\mathbb{N}$ and $y_{k}^{\ast}\in F_{k}^{\ast}$,
\begin{align*}
(AB_{m}^{\rho}(A)(x_{1}^{\ast\ast},\ldots,x_{m}^{\ast\ast})\circ \psi_{1}\circ i_k)(y_{k}^{\ast})&=(AB_{m}^{\rho}(A)(x_{1}^{\ast\ast},\ldots,x_{m}^{\ast\ast})\circ \psi_{1})(0,\ldots,0,y_{k}^{\ast},0,\ldots)\\
&=AB_{m}^{\rho}(A)(x_{1}^{\ast\ast},\ldots,x_{m}^{\ast\ast})(\psi_{1}(0,\ldots,0,y_{k}^{\ast},0,\ldots))\\
&=AB_{m}^{\rho}(A)(x_{1}^{\ast\ast},\ldots,x_{m}^{\ast\ast})(y^{\ast}_{k}\circ \pi_{k})\\&=AB_{m}^{\rho}(\pi_{k}\circ A)(x_{1}^{\ast\ast},\ldots,x_{m}^{\ast\ast})(y^{\ast}_{k}),
\end{align*}
where the last equality follows from Remark \ref{remmer}. This shows that  $AB_{m}^{\rho}(A)(x_{1}^{\ast\ast},\ldots,x_{m}^{\ast\ast})\circ \psi_{1}\circ i_k =AB_{m}^{\rho}(\pi_{k}\circ A)(x_{1}^{\ast\ast},\ldots,x_{m}^{\ast\ast})$ for every $k\in\mathbb{N}$. Therefore,
\begin{align*}
(\psi\circ AB_{m}^{\rho}(A))&(x_{1}^{\ast\ast},\ldots,x_{m}^{\ast\ast})=\psi( AB_{m}^{\rho}(A)(x_{1}^{\ast\ast},\ldots,x_{m}^{\ast\ast}))\\
&=(\psi_{2}^{-1}\circ \psi_{3}^{-1})( AB_{m}^{\rho}(A)(x_{1}^{\ast\ast},\ldots,x_{m}^{\ast\ast}))=\psi_{2}^{-1}(\psi_{3}^{-1}( AB_{m}^{\rho}(A)(x_{1}^{\ast\ast},\ldots,x_{m}^{\ast\ast})))\\&=\psi_{2}^{-1}( AB_{m}^{\rho}(A)(x_{1}^{\ast\ast},\ldots,x_{m}^{\ast\ast})\circ \psi_{1})=(AB_{m}^{\rho}(A)(x_{1}^{\ast\ast},\ldots,x_{m}^{\ast\ast})\circ \psi_{1}\circ i_n)_{n=1}^{\infty}\\&=(AB_{m}^{\rho}(\pi_{n}\circ A)(x_{1}^{\ast\ast},\ldots,x_{m}^{\ast\ast}))_{n=1}^{\infty}.
\end{align*}
By (\ref{mult2}),
\begin{align*}
|(\psi\circ AB_{m}^{\rho}(A))&(x_{1}^{\ast\ast},\ldots,x_{m}^{\ast\ast})|=|(AB_{m}^{\rho}(\pi_{n}\circ A)(x_{1}^{\ast\ast},\ldots,x_{m}^{\ast\ast}))_{n=1}^{\infty}|\\
&=(|AB_{m}^{\rho}(\pi_{n}\circ A)(x_{1}^{\ast\ast},\ldots,x_{m}^{\ast\ast})|)_{n=1}^{\infty}=(AB_{m}^{\rho}(\pi_{n}\circ A)(|x_{1}^{\ast\ast}|,\ldots,|x_{m}^{\ast\ast}|))_{n=1}^{\infty}\\
&=(\psi\circ AB_{m}^{\rho}(A))(|x_{1}^{\ast\ast}|,\ldots,|x_{m}^{\ast\ast}|).
\end{align*}
Thus far we have proved that $\psi\circ AB_{m}^{\rho}(A)$ is a Riesz multimorphism. Since $\psi$ is a positive bijection with positive inverse, $\psi$ is a Riesz homomorphism \cite[Theorem 2.15]{positiveoperators}, so is $\psi^{-1}$. It follows that $AB_{m}^{\rho}(A)=\psi^{-1}\circ \psi \circ AB_{m}^{\rho}(A)$ is a Riesz multimorphism as well. 

The case of $(\oplus_nF_{n})_{0}$-valued multimorphisms is analogous using the lattice  isomorphisms  
\begin{align*}
&\gamma_{1}\colon (\oplus_n F_{n}^{\ast})_{1}\longrightarrow (\oplus_n F_{n})_{0}^{\ast}~,~ \gamma_{1}((y_{n}^{\ast})_{n=1}^{\infty})((y_{n})_{n=1}^{\infty})=\sum_{n=1}^{\infty}y_{n}^{\ast}(y_{n}),\\
&\gamma_{2}\colon (\oplus_n F_{n}^{\ast\ast})_{\infty}\longrightarrow (\oplus_n F_{n}^{\ast})_{1}^{\ast}~,~ \gamma_{2}((y_{n}^{\ast\ast})_{n=1}^{\infty})((y_{n}^{\ast})_{n=1}^{\infty})=\sum_{n=1}^{\infty}y_{n}^{\ast\ast}(y_{n}^{\ast}),
\end{align*}
where $(\oplus_n F_{n}^{\ast\ast})_{\infty}$ is the Banach lattice formed by bounded sequences.
\end{proof}

\begin{example}\rm Let $1 < p < \infty$,  $(p_n)_{n=1}^\infty \subseteq (1, + \infty)$ and let $(F_n)_{n=1}^\infty$ be a sequence of Banach lattices such that, for every $n$, $F_n = c_0$ or $\ell_{p_n}$ or any of the Banach lattices in Example  \ref{ffex}. By Corollary \ref{corohmcve3} and Proposition \ref{llpro}, all Aron-Berner extensions of any Riesz multimorphisms taking values in $\left(\oplus_n F_n\right)_0$ or in $\left(\oplus_n F_n\right)_p$ are Riesz multimorphisms. 
In particular, all Aron-Berner extensions of any Riesz multimorphism taking values in one of the following Banach lattices are Riesz multimorphisms: $c_0(E)$ and $\ell_p(E)$ where $1 < p < \infty$ and $E = c_0, \ell_s, 1< s < \infty, T^*, T, S, d_*(w,1)$ or a projection band in any of these spaces. The space $c_0(c_0)$ is not a new example because it is Riesz isometric to $c_0$.

\end{example}

\bigskip

\noindent Faculdade de Matem\'atica~~~~~~~~~~~~~~~~~~~~~~Instituto de Matem\'atica e Estat\'istica\\
Universidade Federal de Uberl\^andia~~~~~~~~ Universidade de S\~ao Paulo\\
38.400-902 -- Uberl\^andia -- Brazil~~~~~~~~~~~~ 05.508-090 -- S\~ao Paulo -- Brazil\\
e-mail: botelho@ufu.br ~~~~~~~~~~~~~~~~~~~~~~~~~e-mail: luisgarcia@ime.usp.br

%
%
%

\end{document}